\documentclass[a4paper,11pt]{article}

\usepackage[utf8]{inputenc}
\usepackage{amsmath, amsthm, amssymb, mathtools}
\usepackage{amsfonts}
\usepackage{bm}
\usepackage{cancel}
\usepackage{color}
\usepackage{subfig}
\usepackage{stmaryrd}
\usepackage{xparse} 
\usepackage{nomencl} 
\usepackage[colorlinks]{hyperref}
\usepackage{cleveref}
\usepackage{empheq} 
\usepackage{enumerate}
\usepackage{cases}
\usepackage{lineno}
\usepackage{fancyhdr}

\usepackage{textcomp} 
\usepackage{graphicx}
\usepackage{float}
\usepackage[export]{adjustbox}
\usepackage{authblk}

\hypersetup{linkcolor=blue}
\makenomenclature 

\numberwithin{equation}{section}

\newcommand{\N}{\mathcal{N}}
\newcommand{\JJ}{\mathcal{J}}

\DeclarePairedDelimiter{\norm}{\lVert}{\rVert}
\DeclarePairedDelimiter{\jumpop}{\llbracket}{\rrbracket}
\DeclareSymbolFont{matha}{OML}{txmi}{m}{it}
\DeclareMathSymbol{\varv}{\mathbf}{matha}{118}

\theoremstyle{definition}

\newtheorem{Theorem}{Theorem}[section]
\newtheorem{Definition}{Definition}[section]
\newtheorem{Lemma}{Lemma}[section]

\newtheorem{Remark}{Remark}[section]

\newcommand{\restc}{\mathbin{\vrule height 1.6ex depth 0pt width
0.13ex\vrule height 0.13ex depth 0pt width 1.3ex}}

\newcommand\restr[2]{{
  \left.\kern-\nulldelimiterspace 
  #1 
  \vphantom{\big|} 
  \right|_{#2} 
  }}

\NewDocumentCommand{\dgal}{sO{}m}{%
  \IfBooleanTF{#1}
    {\dgalext{#3}}
    {\dgalx[#2]{#3}}%
}

\NewDocumentCommand{\dgalext}{m}{%
  \sbox0{%
    \mathsurround=0pt 
    $\left\{\vphantom{#1}\right.\kern-\nulldelimiterspace$%
  }%
  \sbox2{\{}%
  \ifdim\ht0=\ht2
    \{\kern-.625\wd2 \{#1\}\kern-.625\wd2 \}%
  \else
    \left\{\kern-.7\wd0\left\{#1\right\}\kern-.7\wd0\right\}%
  \fi
}

\NewDocumentCommand{\dgalx}{om}{%
  \sbox0{\mathsurround=0pt$#1\{$}%
  \sbox2{\{}%
  \ifdim\ht0=\ht2
    \{\kern-.625\wd2 \{#2\}\kern-.625\wd2 \}%
  \else
    \mathopen{#1\{\kern-.7\wd0 #1\{}
    #2
    \mathclose{#1\}\kern-.7\wd0 #1\}}
  \fi
}

\newcommand{\R}{{\mathbb{R}}}

\newcommand{\ph}{\varphi}

\newcommand{\PP}{\mathrm{P}}

\crefformat{section}{\S#2#1#3} 
\crefformat{subsection}{\S#2#1#3}
\crefformat{subsubsection}{\S#2#1#3}

\numberwithin{equation}{section}

\makeatletter
\newcommand*\titleheader[1]{\gdef\@titleheader{#1}}
\AtBeginDocument{%
  \let\st@red@title\@title
  \def\@title{%
    \bgroup\normalfont\large\centering\@titleheader\par\egroup
    \vskip1.0em\st@red@title}
}
\makeatother
\title{\bf A class of Discontinuous Galerkin methods for  nonlinear variational problems}
\titleheader{Accepted in \textit{ Mathematics of Computation, \textbf{AMS}},  October  2024}

\author[1,2]{G.\ Grekas}
\author[3]{K.\ Koumatos}
\author[1,3,4]{C.\ Makridakis}
\author[5]{A.\ Vikelis}
\affil[1]{\footnotesize Institute of Applied \& Computational Mathematics, Foundation for Research \& Technology-Hellas.}
\affil[2]{\footnotesize Aerospace Engineering and Mechanics, University of Minnesota, Minneapolis, USA.}
\affil[3]{Department of Mathematics, University of Sussex.}
\affil[4]{\footnotesize Department of Mathematics and Applied Mathematics, University of Crete.}
\affil[5]{\footnotesize Faculty of Mathematics, University of Vienna.}

\date{}

\begin{document}

\maketitle


\begin{abstract}


In the context of Discontinuous Galerkin methods, we study approximations of nonlinear variational problems.  We propose element-wise nonconforming finite element methods to discretize the continuous minimisation problem. Using $\Gamma$-convergence arguments we show that for convex energies the discrete minimisers converge to a minimiser of the continuous problem as the mesh parameter tends to zero, under the additional contribution of appropriately defined penalty terms at the level of the discrete energies. In addition, we provide error analysis in the case where the solution of the continuous minimization problem is   smooth enough. Under appropriate assumptions we show    local existence, uniqueness and corresponding optimal order error estimates for the discrete Euler-Lagrange equations of our scheme.  
We finally substantiate the feasibility of our methods by numerical examples. 

\end{abstract}


\section{Introduction}
In this paper we propose and analyse  Discontiuous Galerkin (DG) methods for nonlinear variational problems of a rather general form. 
We formulate  Discontiuous Galerkin methods in a systematic and consistent way and we analyse the convergence properties of the discretisations 
by appropriately restricting the energy class. 
Indeed, we consider the problem of minimising the total  energy
\begin{equation}\label{em1}
\begin{aligned}
\mathcal {E} [u] = \int_{\Omega} W(\nabla u( x))  dx - \int_{\Omega} f (x)\,  u(x)  dx  ,
\end{aligned}
\end{equation}
with $u $ in appropriate Sobolev spaces incorporating boundary conditions.
$W$ is the strain energy function and $f$ a given external field. Local minima of this problem satisfy appropriate Euler-Lagrange equations when certain smoothness requirements are met which are nonlinear.  Despite the progress in Discontinuous Galerkin methods, appropriate formulations of the  nonlinear principal part of equations of the form \eqref{em1} are, to the best of our knowledge,  quite limited. 
This paper is devoted to the systematic derivation of such formulations and their mathematical analysis. 

Problems in the calculus of variations seeking to minimise functionals of the form 
\eqref{em1} are important and arise in several different applications. Nevertheless,  many computational challenges remain largely open. The analytical issues are quite subtle and the numerical analysis, even at the scheme design level, requires new approaches. 

Our focus in the present paper is the design of methods that are consistent with structural properties. We theoretically verify the convergence of the methods by considering two alternative approaches. First, without smoothness assumptions on the minimisers,  we show that the $\Gamma$-convergence framework is applicable, considering as a working hypothesis that $W$ is convex. Subsequently, we provide error analysis in the case where the solution of the continuous minimisation problem is assumed to be  smooth enough.  Finally we provide test examples showing that the theoretical predictions are indeed experimentally confirmed.  \\

\noindent
\emph{Our Contributions.\/}
 We start in  Section \ref {DG_design}  by systematically defining the discrete method.
 Our approach recovers the standard interior penalty method in the case 
 of linear elliptic problems. Its final form is given at the energy 
 level, see \eqref{en_dg_1}, and includes terms involving energies of elementwise gradients and explicit jump terms, as is typical with discontinuous Galerkin methods. A penalty term is included to ensure coercivity and therefore discrete stability.  In Section \ref {DG_design}  we present other available DG methods for the approximation of \eqref{em1}.
 Notably, these methods either are defined implicitly through lifting operators, or are defined on the first variations and lack a parent approximate energy. Our method fills this gap in the literature, since it is designed elementwise;   thus its structure follows the typical paradigm of DG methods and in addition it is defined at the energy level approximating \eqref{em1}. 
 
 Section \ref{Convergence} is devoted to the analysis of the method using the tool of $\Gamma-$conver\-gence in the present DG setting. We show that the discrete energies $\Gamma-$con\-verge to \eqref{em1} as the mesh parameter $h$ tends to zero. In the proofs we assume that the energy $W$ is convex and satisfies certain growth conditions, \eqref{eq:growth}. 
 Then we show the convergence of discrete  minimisers to the solution of the original problem provided that the penalty terms are appropriately chosen. Towards this goal,   equi-coercivity, the $\liminf$ and the
$\limsup$ inequalities are derived  for the  discrete energy functional. As in other works, \cite{buffa2009compact, di2010discrete}, 
related to $\Gamma$-convergence analysis for DG methods, the analysis is challenging due to the fact 
that discrete functions are totally discontinuous while the space of continuous minimisers is a 
subspace of $H^1\, .$  See \cite{bartels2017bilayer, BNNtogkas:2021, grekas2022approximations} for recent $\Gamma$-convergence results using conforming or non-conforming finite elements approximating nonlinear fourth order problems. 
Although we follow, \cite{buffa2009compact, di2010discrete}, in reconstructing discrete gradients, in  key technical compactness and lower semicontinuity arguments, our analysis 
 introduces new ideas to tackle the problem at hand.  As a consequence, we show that the element-wise terms introduced in our formulation, indeed contribute to the consistency and stability of the method, and furthermore, the present analysis   reveals  the subtle form of the penalty terms which are required for convergence.  

 Section 4 is devoted to error analysis in the case where the solution of the continuous minimization problem is assumed to be smooth enough. We show that fixed point arguments initially introduced for the dynamic problem and then applied to the stationary Euler-Lagrange equations, \cite {makr_1993, ortner2007discontinuous},
 are applicable to our scheme. Such arguments yield local existence and uniqueness and corresponding optimal order error estimates for the discrete Euler-Lagrange equations.  An interesting new feature of our analysis is that we are able to complete the proofs by  assuming only natural G\aa rding-type inequalities for the stationary problem, thereby extending the analysis of  \cite {ortner2007discontinuous} in which the analysis of the elliptic problem was based  on strong  coercivity assumptions, see Section 4 for details.

We conclude with numerical experiments in Section~\ref{computations}, showing that the proposed method can lead to reliable predictions of the problem at hand and is in agreement with the theoretical results. 

\def\<{{\langle }}
\def\>{{\rangle }}

\newcommand{\J}[1]{{  [\! [#1]\! ]  }}   

\newcommand{\Av}[1]{{  \{\!\!\{#1\}\!\!\}  }}

\section{Design of the DG method}
\label{DG_design}

To motivate the design of the methods we consider first 
a smooth stored energy function $W$ and assume that $f=0.$ 
We consider the minimisation problem on an appropriate Banach space $X$ (typically a Sobolev space) containing functions 
$u \colon \R^d\supseteq \Omega  \rightarrow \R^N$
\begin{equation}
\begin{aligned}
  & \text{find a   minimiser    in $ {X} $ of : }\\
 &\mathcal{E} (u) = \int _\Omega \,  W (\nabla u  ) \, dx \, . 
\end{aligned}
\end{equation}
Temporarily, to avoid problems with boundary conditions, we assume that  elements of $ {X} $  are of the form $Fx +v (x)$ where $v\in V,$ a subspace of $X$ with periodic boundary conditions, $F \in R^{N \times d}$ and $\Omega = [0, 1] ^d$.  
If such a minimiser exists, then employing the summation convention, 
\begin{equation*}
\begin{aligned}
   \< D{\mathcal{E} } (y), v\>  
   =      \int_\Omega  S_{i\alpha} (\nabla y(x))   \,  \partial_\alpha  v^i  (x) \, dx =   0,  \qquad   v\in V \, ,
\end{aligned}
\end{equation*}
 where the Piola-Kirchhoff stress  tensor is defined through
\begin{equation*}
\begin{aligned}
&S_{i \alpha} (\zeta)  :=   \frac{\partial W (\zeta ) } {\partial  {\zeta_{i\alpha}}} \, .
\end{aligned}
\end{equation*}
We will use the notation 
\begin{equation*}
\begin{aligned}
&S = \partial _\zeta \, W\, .
\end{aligned}
\end{equation*}
We observe now that 
\emph{a key structural property }
of the above problem is that \emph{affine maps of the form $Fx + b$ are critical points of the energy functional.} Indeed, 
the energy $ \mathcal{E}   $ satisfies 
\begin{equation*}
 \< D\mathcal{E}   (y_F), v\>=  0, \quad y_F (x) = Fx\, , \end{equation*}
for all    $v\in $ as  the  Gauss-Green Theorem implies that
  \begin{equation*} \begin{split}
 \< D\mathcal{E}   (y_F), v\> =   
\partial _\zeta \, W (  F ) \cdot \int _\Omega  \nabla v dx&  = S  ( F )  \cdot \int _\Omega  \nabla v dx\\
& = S_{i\alpha} ( F ) \int _{\partial \Omega}   \,  n_\alpha  v^i  dx = 0\, 
 \end{split}\end{equation*}
 due to periodicity in $X.$
We would like our numerical methods to satisfy a discrete analogue of this fundamental property.  It is 
easy to see that for standard $C^0$ finite elements the above argument goes through and  $y_F$ is a 
critical point of the discrete functional.  However, we would like to consider the case of 
discontinuous discrete spaces. 
Among other reasons, standard finite elements result in a method which might be ``too rigid" in certain applications due to the continuity requirement.  
The discrete version of the above property was called \emph{patch test consistency} in the early literature of finite elements. 
The next observation is crucial when considering discontinuous elements. 
Assume that $\Omega $ is split in two pieces $\Omega _1 $ and $ \Omega _2 ,$   their interface is denoted by $\Gamma$ and the test function $v$ has a possible
 discontinuity at $\Gamma .$
Then 
\begin{equation*} \begin{split}
 \< D\mathcal{E}   (y_F), v\> =&   
\partial _\zeta \, W ( F ) \cdot \int _{\Omega _1 } \nabla v dx + \partial _\zeta \, W ( F ) \cdot \int _{\Omega _2}  \nabla v dx  \\
 & = S  ( F )  \cdot \left [ \int _{\Omega _1 } \nabla v dx +  \int _{\Omega _2} \nabla v dx\right ] \\
& = S_{i\alpha} ( F ) \left [ \ \int _{\partial \Omega_1}   \,  n_{\Omega_1, \, \alpha} v^i  dS + \int _{\partial \Omega_2}   \,  n_{\Omega_1, \, \alpha}  v^i  dS \right ]   \\
& = S_{i\alpha} ( F ) \left [ \ \int _{\partial \Omega }   \,  n_\alpha  v^i  dS + \int _{\Gamma}   \, \Big [ n_{\Omega_1, \, \alpha}  v^i  + n_{\Omega_2, \, \alpha}  v^i  \Big ] dS \right ]\\
& = S_{i\alpha} ( F )  \int _{\Gamma}   \, \Big [ n_{\Omega_1, \, \alpha}  v^i  + n_{\Omega_2, \, \alpha}  v^i  \Big ] dS \\
&= S  ( F )  \cdot  \int _{\Gamma}\  \J {v\otimes n} \, dS.
 \end{split}\end{equation*}
It is clear that the discontinuity of the function $v$ induces a violation of the key property $  \< D\mathcal{E}   (y_F), v\>=0\, .$
On the other hand, this calculation motivates a new definition, see \cite{makridakis2014atomistic}, of the ``discrete" functional which we call discontinuous interface coupling on  $\Gamma : $ 
    \begin{equation*} 
    \begin{split}
  \mathcal{E}_{dg}   (y ) = &\int _{\Omega _1  } \, W( \nabla y ) dx + \int _{\Omega _2  } \, W( \nabla y ) dx   \\  
 &\quad  -  \int _{\Gamma}\ S  (\Av{ \nabla y}  )  \cdot \J {y \otimes n} \, dS\, .
\end{split}
 \end{equation*}
The new energy is consistent. Indeed, it is a simple matter to check that   
 \begin{equation*} 
 \int _{\Gamma} D \big [   S (\Av{\nabla w } ) {\cdot}  \J {w  \otimes n}  \, \big ]\,  v  =
  \int _{\Gamma}   {\Bigl(}
  \partial _\zeta S (\Av{\nabla w  } )\, \J {w\otimes n  }  {\Bigl) \cdot}
   \Av{\nabla v }   
+  S(\Av{\nabla w  } ) {\cdot} \J {v\otimes n} ,
 \end{equation*}
which  reduces to  $  {S(F  )\cdot \J {v \otimes n } }$  when $w= y_F.$ Repeating then the arguments above we conclude that $  \< D\mathcal{E} _{dg}  (y_F), v\>=0\, .$
 
\subsection{New DG finite element energy minimisation methods. } Next, motivated by the above discussion we shall define the discrete energy functional on appropriate 
discontinuous Galerkin spaces. First we introduce the necessary notation, before the DG method is defined. The section concludes with some remarks on its relationship to 
other known methods for energy minimisation as well as  to classical DG methods for linear problems.

 \subsubsection{Finite element spaces and discrete gradients} 
 For simplicity, let us assume 
 that $\Omega$ is a polytope domain, and consider a triangulation $T_h$ of $\Omega$ of mesh size $h>0$. Next we will require the partitions of the domain
to be shape regular \cite{brenner2007mathematical}, i.e., there exists $c > 0$ such that
$$\rho_K \geq 
\frac{h_K}{c}
 \text{ for all } K \in  T_h,
$$
where $\rho_K$ is the diameter of the largest ball inscribed in $K$. To simplify matters, we assume that no hanging nodes are present. 
The methods and proofs can be extended to allow   hanging nodes under appropriate assumptions.

 We seek discontinuous Galerkin approximations to \eqref{eq:energy}. To this end, we introduce the space of piecewise polynomial functions
\begin{equation}\label{eq:DGspace}
V^q_h(\Omega):=\left\{v\in L^\infty(\Omega)\,:\,\restr{v}{K}  \in \mathbb{P}_q(K), K \in T_h \right\},
\end{equation}
where $\mathbb{P}_q(K)$ denotes the space of polynomial functions on $K$ of degree $q\in \mathbb{N}$ which is henceforth fixed. Moreover, we adopt the following standard notation:
the \textit{trace} of 
functions in $V^q_h(\Omega)$ belongs to the space
\begin{align}
 T(E_h) :=  \Pi_{e\in E_h} \mathbb{P}_q(e),
\end{align}
where $E_h$ is the set of mesh edges.
The \textit{average} and \textit{jump} operators over
$T(E_h)$ are defined by:
\begin{equation}
 \begin{aligned}
\dgal{ \cdot } \colon & T(E_h) \mapsto L^p(E_h) \\
  &\dgal{ w } := \left\{
        \begin{array}{ll}
             \frac{1}{2} (\restr{w}{K_{e^+} }  + \restr{w}{K_{e^-} }),
             & \quad \text{for } e \in E_h^i \\
             w, & \quad \text{for } e \in E_h^b,
        \end{array}
    \right.
 \end{aligned}
\label{average_operator}
\end{equation}
\begin{equation}
 \begin{aligned}
 \jumpop{ \cdot }
 \colon& T(E_h) \mapsto L^p(E_h) \\
  &\jumpop{ v }	 := \restr{v}{K_{e^+} }  - \restr{v}{K_{e^-} },
             & \quad \text{for } e \in E_h^i .
 \end{aligned}
\label{jump_operator}
\end{equation}
We often apply the jump operator on dyadic products $v\otimes n$, for $v \colon \Omega \rightarrow \R^N$, $n\in \R^d$ and we write
\begin{equation}
 \begin{aligned}
  &\jumpop{ v\otimes n_e }	 := \restr{v\otimes n_{e^+}}{K_{e^+} }  + \restr{v\otimes n_{e^-}}{K_{e^-} },
             & \quad \text{for } e \in E_h^i  
 \end{aligned}
\label{jump_operator_tensor}
\end{equation}
where $K_{e^+}$, $K_{e^-}$ are the elements that share the internal edge $e$ with corresponding outward pointing normals
$n_{e^+}, n_{e^-}$, and $E_h^i$, $E_h^b$ denote respectively the set of internal and boundary edges.

The space $V^q_h(\Omega)$ is equipped with the norm
\[
\|v\|^p_{W^{1,p}(\Omega,T_h)} = \|v\|^p_{L^p(\Omega)} + |v|^p_{W^{1,p}(\Omega,T_h)},
\]
where the seminorm $|\cdot|_{W^{1,p}(\Omega,T_h)}$ is defined by
\begin{equation}\label{eq:seminorn}
|v|^p_{W^{1,p}(\Omega,T_h)} := \sum_{K\in T_h} \int_K |\nabla v|^p +\sum_{e\in E_h^i} \frac{1}{h^{p-1}}\int_{e} |\jumpop{v}|^p.
\end{equation}

We note that in the above expression $\nabla v$ denotes the gradient, in the classical sense, of the polynomial function $v$ over the set $K$. As a function in $SBV(\Omega)$, i.e. in the class of special  functions of bounded variation \cite{ambrosio2000functions},
$v\in V^q_h(\Omega)$, admits a distributional gradient which is a Radon measure without a Cantor part and decomposes as
\begin{align}
Dv = \nabla v\mathcal{L}^d\restc\Omega + \jumpop{v}\otimes n_{e} \mathcal{H}^{d-1}\restc E_h
\label{distrib_gradient}
\end{align}
where $\nabla v$ denotes the approximate gradient. To avoid confusion, we denote the above approximate gradient by $\nabla_h v$, i.e. we write $\nabla_h: V^k_h \to V^{k-1}_h$ for the operator defined by
\[
\restr{\left(\nabla_h v\right)}{K} = \nabla (\restr{v}{K}).
\]
In particular, it holds that
\[
\int_{\Omega} W(\nabla_h u_h) = \sum_{K\in T_h} W(\nabla u_h).
\]
A crucial role in the analysis is played by the lifting operator defined below.

 \begin{Definition}
\label{def:lift_operator}
The (global) lifting operator $R_h : T(E_h) \to V^{q-1}_h(\Omega)$ is defined by its action on elements of $V^{q-1}_h(\Omega)$ as
\begin{align}
 \int_\Omega R_h(\varphi) : w_h \,dx =\sum_{e\in E^i_h} \int_e \dgal{ w_h } : 
\jumpop{ \varphi \otimes n_e } \, ds, \quad  \forall w_h \in V^{q-1}_h(\Omega).
\label{eq:lift_operator}
\end{align}
\end{Definition}
Through the lifting operator we may also define the discrete gradient $G_h$ for $u_h\in V^q_h$ as
\begin{align}
G_h(u_h) = \nabla_h u_h - R_h(u_h ).
\label{discrete_gradient}
\end{align}
In general, the lifting operator maps into polynomials of degree $l \ge0 $, \cite[Section 2.3]{di2010discrete}. For convenience we choose $l=k-1$ although this is not essential.

\subsubsection{Discrete energy functionals}
 Motivated by the  discussion in the beginning of this section,  one  is led to define the discrete energy over totally  discontinuous finite element spaces through
\begin{equation}\label{en_dg_1}
\begin{aligned}
\mathcal{E}_{dg} [u_h] &= \sum_{K \in T_h}\int_{K}  W(\nabla u_h(x)) \\ 
 & -\sum_{e \in E^i_h}\int_{e}\ S  (\Av{ \nabla u_h}  )  \cdot \J {u_h \otimes n_e}  ds
 + \text{Pen}\, ( {u_h} ),
 \end{aligned}
\end{equation}
where as is typical with discontinuous Galerkin formulations $\text{Pen}\, ( {u_h} )$ is an appropriate penalty term, i.e., a functional which penalises the jumps $ \J {u_h} $ and it will be specified in the sequel.  

Consider now $\PP: L^2(\Omega) \to V^{q-1}_h$ the standard (local) $L^2$ projection so that for $V \in L^2(\Omega)$
\[
\int_\Omega \left(\PP V\right) : w = \int_{\Omega} V :  w,\quad\forall\,w\in V^{q-1}_h.
\]
In particular, since $\nabla_h u_h$ is in the finite-dimensional space, it holds that $DW(\nabla_h u_h) \in L^2(\Omega)$. 
We can consider a slight variation of the energy \eqref{en_dg_1}, including the projection operator on the interface terms, 
\begin{equation}\label{en_dg_2}
\begin{aligned}
\mathcal{E}_{h} [u_h] &= \sum_{K \in T_h}\int_{K}  W(\nabla u_h(x)) \\ 
 &-\sum_{e \in E^i_h}\int_{e}\  \Av{ \PP\, S  (\nabla u_h  )}  \cdot \J {u_h \otimes n_e}  ds 
+ \text{Pen}\, ( {u_h} ).
 \end{aligned}
\end{equation}
The energy \eqref {en_dg_2} is still consistent in the sense discussed in the beginning of 
Section~\ref{DG_design}  since $\Av{ \PP\, S  ( F  )} =  S  (\Av{ F}  ) = S ( F\, )$ given that $F $ is constant.    For piecewise linear elements ($q=1$)  there holds $\PP\, S  (\nabla u_h  )= S  (\nabla u_h  )\, .$ In addition, for high-order elements, notice that $\PP v$  can be computed elementwise due to the discontinuity of the functions of $ V^{q-1}_h$ across elements. Thus the increase in the computational costs
of \eqref{en_dg_2} compared to \eqref{en_dg_1} is minimal.

\subsubsection{Relationship to known discrete functionals} The following functional can be defined on $ V^{q}_h:$
\begin{equation}\label{en_dg_lo}
\begin{aligned}
\mathcal{E}_{lo} [u_h] &= \sum_{K \in T_h}\int_{K}  W(\nabla u_h(x)  - R_h(u_h ))  
+ \text{Pen}\, ( {u_h} )\\
&= \sum_{K \in T_h}\int_{K}  W(G_h(u_h))  
+ \text{Pen}\, ( {u_h} ).
 \end{aligned}
\end{equation}
This consists a natural choice, given that $G_h(u_h) = \nabla u_h(x)  - R_h(u_h )$ is considered as the discrete gradient.  
This method was introduced by Ten Eyck and Lew (2006) \cite{ten2006discontinuous} and further analysed in  
\cite{ortner2007discontinuous}, \cite{buffa2009compact}. 
Although lifting operators are a typical analysis tool for DG methods, in the present context,  the evaluation of this functional requires the
 computation of $R_h(u_h )$ (and thus of  $G_h(u_h)$) which appears implicitly  in $W.$
Computationally this reduces to the solution of local problems only, and it is not expensive.
Early contributions on the relevance of DG methods for nonlinear variational problems include
\cite{gobb_prohl_1999}. In this work the discrete functional involved only jump terms on the interface and thus for quadratic potentials the discrete energy is not consistent with  
typical DG formulations. 
It is interesting that functionals such as   \eqref{en_dg_1}, \eqref{en_dg_2} including typical 
consistent discontinuous Galerkin terms in explicit form and   being at the same time  appropriate for nonlinear energy minimisation problems were not available up to now.

\subsubsection{Euler-Lagrange equations}  The first variation  of \eqref{en_dg_1} is 
\begin{equation}\label{D_en_dg_0}
\begin{aligned}
\< D\mathcal{E}_{dg} [u_h] , v\> &= \sum_{K \in T_h}\int_{K}  DW(\nabla u_h(x)) \nabla v_h(x)\\ 
 & -\sum_{e \in E^i_h}\int_{e}\ S  (\Av{ \nabla u_h}  )  \cdot \J {v_h \otimes n_e}  ds
 \\
 &-\sum_{e \in E^i_h}\int_{e}\ 
\Bigl (   \partial _\zeta S (\Av{\nabla u_h  } )\, \J {u_h  \otimes n_e }  {\Bigl) \cdot}
   \Av{\nabla v_h }   
   ds \\
&   + \< D \text{Pen}\, ( {u_h} ), v_h \>,
 \end{aligned}
\end{equation}
where $\partial _\zeta S (\Av{\nabla u_h  } )\, \J {u_h  \otimes n_e }  {\Bigl) \cdot}
   \Av{\nabla v_h } = $ $(\partial _\zeta S (\Av{\nabla u_h  } )_{ij, kl} v_{h_{k,l}} u_{h_i} n_{e_j}$.
A similar relationship holds for \eqref{en_dg_2}. Notice that a similar form was considered in \cite{ortner2007discontinuous}
 which was designed directly to approximate the Euler-Lagrange equations arising in nonlinear elasticity in the DG framework. 
In this work, compared to \eqref{D_en_dg_0} the term involving $  \partial _\zeta S$ was not present, but it was not clear if the resulting 
form was indeed a first variation of a discrete functional, and thus the association to an energy minimisation problem was missing. 
In the case of the standard potential $ W(F) = \frac 1 2 |F|^2$ corresponding to the Laplace equation, 
 \eqref {D_en_dg_0} reduces to the well known  interior penalty method, see e.g., \cite{Arnold_DG, brenner2007mathematical},
\begin{equation}\label{D_en_dg_1}
\begin{aligned}
\< D\mathcal{E}_{dg} [u_h] , v\> &= \sum_{K \in T_h}\int_{K}   \nabla u_h(x) \nabla v_h(x)\\ 
 &-\sum_{e \in E^i_h}\int_{e}\  \Av{ \nabla u_h}     \cdot \J {v_h \otimes n_e}  
  + \J {u_h  \otimes n_e}   \cdot \Av{\nabla v_h } ds \\
&   + \< D \text{Pen}\, ( {u_h} ), v_h \> .
 \end{aligned}
\end{equation}

\section{Convergence of Discrete  Minimizers}\label{Convergence}
For $\Omega\subset \R^d$ a bounded, polytope domain, consider the minimisation problem
\begin{equation}\label{eq:energy}
\mathcal{E}(u) = \int_\Omega W(\nabla u(x))\,dx,\quad u:\Omega\to \R^N,\,\,u = u_0,\mbox{ on }\partial\Omega.
\end{equation}
We assume throughout that $W:\R^{N\times d}\to\R$ is of class $C^1$  and satisfies
\begin{equation}\label{eq:growth}
-1 + |\xi|^p \lesssim W(\xi) \lesssim 1 + |\xi|^p, \quad p>1.
\end{equation}
Hence, we naturally pose the problem of minimising \eqref{eq:energy} over the 
set of admissible functions
\[
\mathbb{A}(\Omega) = \left\{u \in W^{1,p}(\Omega)^N\,:\, u = u_0\,\,\mathcal{H}^{d-1}\mbox{ a.e. on $\partial\Omega$}\right\}.
\label{eq:cont_fun_space}
\]
Given \eqref{eq:growth}, it is known that \eqref{eq:energy} admits minimisers of class $W^{1,p}_{u_0}(\Omega)$ as long as $W$ is quasiconvex.


Employing the aforementioned ideas for the discretization of energy functionals in the DG setting, see eq.~(\ref{en_dg_2}),  and applying Dirichlet boundary conditions through Nitsche's 
approach, for $u_h\in V^q_h$  our discrete energy functional has the form
\begin{equation}\label{eq:energy_discrete}
\begin{aligned}
\mathcal{E}_h(u_h) & = \sum_{K\in T_h} \int_K W(\nabla u_h) dx  - \sum_{e\in E^i_h}\int_e \dgal{{\PP} DW(\nabla u_h)}: \jumpop{u_h\otimes n_e}\,ds \\
&\qquad + \alpha \, {\rm Pen}(u_h),
\end{aligned}
\end{equation}
where $\alpha$ will be chosen large enough and the penalty ${\rm Pen}(u_h)$ is defined as 
\begin{equation}\label{eq:penalty}
\begin{aligned}
{\rm Pen}(u_h) & := \left(1 + |u_h|^{p}_{W^{1,p}(\Omega,T_h)} \right)^{\frac{p-1}{p}} \left(\sum_{e\in E_h} \frac{1}{h_e^{p-1}}\int_e |\jumpop{u_h}|^p \right)^{\frac1p},
\end{aligned}
\end{equation}
where for the boundary faces we use the notation
\begin{align}
\jumpop{u_h} =  u_h^- -u_0, \text{ for } e \in E_h^b = E_h \cap \partial \Omega,
\end{align}
and $u_0$ are the boundary conditions encoded in the continuous function space $\mathbb{A}(\Omega)$ of eq.~(\ref{eq:cont_fun_space}).
The above penalty term seems complicated but it is an appropriate $L^p$ modification of the typical $L^2$ penalty
\[
\sum_{e\in E^i_h} \frac{1}{h_e}\int_e |\jumpop{u_h}|^2.
\]
 As it will be evident from the proofs, by assuming that $W \in C^2$ and 
 by imposing further restrictions on the growth of its derivatives, one may also choose 
 the penalty term as
\begin{equation}\label{eq:penalty2}
\begin{aligned}
{\rm Pen}(u_h) & := \left(1 + |u_h|^{p-2}_{W^{1,p}(\Omega,T_h)} \right) \left(\sum_{e\in E_h} \frac{1}{h_e^{p-1}}\int_e |\jumpop{u_h}|^p \right)^{\frac2p}
\end{aligned}
\end{equation}
which reduces precisely to the standard $L^2$ penalty term for $p=2$.
We note that the above penalty terms are required to absorb terms in the proofs of  Lemma \ref{lemma:liminf} and Lemma \ref{lemma:limsup} which appear due to the nonlinearity of $W$ and cannot be 
handled by the typical $L^2$ penalty term.

Recall that $DW(\nabla_h u_h) \in L^2(\Omega)$. Also, as $R_h(u_h) \in V^{q-1}_h$, we find that
\[
\begin{aligned}
\int_{\Omega} R_h(u_h) : DW(\nabla_h u_h) & = \int_{\Omega} R_h(u_h) : \PP DW(\nabla_h u_h) dx\\
& =  \sum_{e\in E^i_h}\int_e \dgal{{\PP} DW(\nabla_h u_h)}: \jumpop{u_h\otimes n_e}\,ds\\
& =  \sum_{e\in E^i_h}\int_e \dgal{{\PP} DW(\nabla u_h)}: \jumpop{u_h\otimes n_e}\,ds.
\end{aligned}
\]
Hence, we can re-write \eqref{eq:energy_discrete} in the compact form
\begin{equation}\label{eq:energy_discrete_compact}
\mathcal{E}_h(u_h) = \int_\Omega \left[W(\nabla_h u_h) - R_h(u_h) : DW(\nabla_h u_h)\right]d x + \alpha\, {\rm Pen}(u_h).
\end{equation}

\subsection{Preliminary results.}


Previously we have defined the lifting operator $R_h$ in eq~(\ref{eq:lift_operator}).
In particular, for $u_h\in V^q_h(\Omega)$ we have that
\[
 \int_\Omega R_h(u_h) : w_h \,dx =\sum_{e\in E^i_h} \int_e \dgal{ w_h } : 
\jumpop{ u_h \otimes n_e } \, ds.
\]
Moreover, we note the following lemma, \cite{buffa2009compact},  providing an estimate of $\|R_h(u_h)\|_{L^p}$ in terms of the $W^{1,p}(\Omega,T_h)$ seminorm.
 \begin{Lemma}[Bound on global lifting operator] 
  \label{prop:Rh_bound}
For all $u_h \in V^q_h(\Omega)$ it holds that
  \begin{align}
   \int_\Omega |R_h(u_h)|^p  \leq 
  C_R \sum_{e \in E_h^i}h_e^{1-p} \int_e |\jumpop{u_h}|^p,
  \label{eq:R_bound}
  \end{align}
  where the constant $C_R$ is independent of $h$.
\end{Lemma}
So far, we have defined $\nabla_h$ as the approximate gradient of the 
distributional gradient, eq.~(\ref{distrib_gradient}). Now the relation of the discrete 
gradient $G_h(u_h)$, eq.~(\ref{discrete_gradient}), to the distributional 
gradient of an SBV function is revealed by the following: for $u_h\in V^q_h$, 
$\varphi\in C^1_c(\Omega)$ and $\Omega = K_1\cup K_2$, we 
compute
\begin{align*}
\langle Du_h,\varphi\rangle & = - \int_\Omega u_h \cdot {\rm div}\,\varphi\,dx \\
& =  \int_{K_1} \nabla u_h : \varphi   + \int_{K_1} \nabla u_h : \varphi   - \int_e \varphi : \jumpop{u_h \otimes n_e} \\
& = \int_\Omega  \varphi : \nabla u_h  - \int_e  \varphi : \jumpop{u_h \otimes n_e} \\
& = \langle \nabla u_h \mathcal{L}^d\restc\Omega +  \jumpop{u_h\otimes n_{e}} \mathcal{H}^{d-1}\restc e , \varphi\rangle.
\end{align*}
That is, $G_h(u_h)$ acts on $\varphi \in V^{q-1}_h(\Omega)$ in the same way that the Radon measure $Du_h$ acts on $\varphi\in C^1_c(\Omega)$ and $G_h(u_h)$ can be interpreted as a discrete version of the distributional gradient of the $SBV$ function $u_h$. Hence, the following lemma comes as no surprise, where  eq.~(\ref{eq:R_bound}) is employed for the proof, \cite[Proposition 2.1]{di2010discrete}.
 \begin{Lemma}
  \label{lemma:discrete_grad}
There exists a constant $C>0$ such that for all $u_h \in V^q_h(\Omega)$ it holds that
\[
\|G_h(u_h) \|_{L^p(\Omega)} \leq C |u_h|_{W^{1,p}(\Omega,T_h)}.
\]
 Moreover, whenever $u_h \rightarrow u$ in 
$L^p(\Omega)$ and $|u_h|_{W^{1,p}(\Omega, T_h)}$ is uniformly bounded with respect to  $h$, then 
  
  \begin{equation}
\begin{aligned}
\lim_{h \rightarrow 0 }\int_\Omega G_h(u_h) : \ph =
 -\int_\Omega u \cdot {\rm div}\,\ph, \quad 
 \forall  \ph \in  C_c^\infty(\Omega).
\end{aligned}
\label{eq:discrete_gradient_distr_convergence}
\end{equation}
In particular, if $u \in W^{1,p}(\Omega)$,  it holds that
\[
G_h(u_h) \rightharpoonup \nabla u\mbox{ in }L^p(\Omega).
\]
\end{Lemma}
Next  Poincar\'{e} type inequalities for DG spaces are presented. Similar results have been shown in \cite{buffa2009compact}.
\begin{Lemma}[Poincar\'{e} inequality for DG spaces.]
\label{thm:Poincare}
Let $u_h \in V_h(\Omega)$ and $\Gamma \subset \partial \Omega$ with $|\Gamma| >0$, then for all $h \le 1$ it holds that
\begin{align}
    \norm{u_h}_{L^p(\Omega)} \lesssim |u_h|_{W^{1,p}(\Omega,T_h)} + \norm{u_h}_{L^p(\Gamma)}.
    \label{eq:Poincare}
\end{align}
\end{Lemma}
\begin{proof}
    Let $w_h$ be the approximation of $u_h$ from the continuous finite element space 
    $V_h(\Omega)\cap W^{1,p}(\Omega)$ through the reconstruction operator $Q_h: V_h \rightarrow W^{1,p}(\Omega)$, i.e. $w_h= Q_h u_h$,
    defined in \cite[Section 3.2]{buffa2009compact}. The
    following inequality results from the local estimates given in \cite[Theorem 3.1]{buffa2009compact}:
    \begin{align}
    \norm{u_h -w_h}_{L^p{(\Omega)}} + \norm{u_h -w_h}_{L^p{(\Gamma)}} + \norm{\nabla w_h}_{L^p{(\Omega)}}\lesssim |u_h|_{W^{1,p}(\Omega, T_h)}.
    \label{eq:rec_estimates}
    \end{align}
     Standard Poincar\'{e} inequalities
    for functions  in $W^{1,p}(\Omega)$, \cite{ziemer1989weakly},
     together with  inequality (\ref{eq:rec_estimates}) give the bound
    \begin{align*}
        \norm{w_h}^p_{L^p(\Omega)}
        &\lesssim  \norm{\nabla w_h}_{L^p(\Omega)}^p + \left|\frac{1}{|\Gamma|} \int_\Gamma w_h ds\right|^p 
        \lesssim |u_h|_{W^{1,p}(\Omega, T_h)}^p 
        \\
        &+ \norm{w_h - u_h}^p_{L^p(\Gamma)} + \norm{u_h}^p_{L^p(\Gamma)} 
        \lesssim |u_h|_{W^{1,p}(\Omega, T_h)}^p + \norm{u_h}^p_{L^p(\Gamma)}
    \end{align*}
Then, it is straightforward to deduce inequality~(\ref{eq:Poincare}) again with the help of eq.~(\ref{eq:rec_estimates})
\begin{align*}
     \norm{u_h}^p_{L^p(\Omega)}  \lesssim \norm{u_h - w_h}^p_{L^p(\Omega)}  + \norm{w_h}^p_{L^p(\Omega)}
      \lesssim |u_h|_{W^{1,p}(\Omega, T_h)}^p + \norm{u_h}^p_{L^p(\Gamma)}.
\end{align*}
\end{proof}

\subsection{$\Gamma$-convergence for convex energies}
We may now state our main result:

\begin{Theorem}
\label{thm:main}
Let $W:\R^{N\times d}\to\R$ be convex and of class $C^1$, satisfying the coercivity and growth condition 
\eqref{eq:growth}. Suppose that the penalty parameter $\alpha$ from \eqref{eq:energy_discrete} is large 
enough and $u_h \in V^q_h(\Omega)$ is a sequence of minimisers of $\mathcal{E}_h$, eq.~\eqref{eq:energy_discrete}, i.e.
\[
\mathcal{E}_h(u_h) = \min_{V^q_h(\Omega)} \mathcal{E}_h.
\]
Then there exists $u\in \mathbb{A}(\Omega)$ such that, up to a subsequence,
\[
u_h \to u\mbox{ in }L^p(\Omega)
\]
and
\[
\mathcal{E}(u) = \min_{\mathbb{A}(\Omega)} \mathcal{E}.
\]
\end{Theorem}

The proof of Theorem \ref{thm:main} is a direct consequence of the fact that $\mathcal{E}_h$ $\Gamma$-converges to $\mathcal{E}$ with respect to the strong topology of $L^p(\Omega)$. The remainder of this section is devoted to the proof of the $\Gamma$-convergence result. We start by proving compactness in an appropriate topology for sequences of bounded energy. 
%

\begin{Lemma}
\label{lemma:compactness}
For $u_h \in V^q_h(\Omega)$ it holds that
\[
\|u_h\|^p_{W^{1,p}(\Omega,T_h)} \lesssim 1 + \mathcal{E}_h(u_h).
\]
In particular, if 
\[
\sup_h\mathcal{E}_h(u_h) \leq C
\]
there exists $u\in \mathbb{A}(\Omega)$ such that, up to a subsequence,
\[
u_h \to u\mbox{ in }L^p(\Omega)
\]
and
\[
G_h(u_h) \rightharpoonup \nabla u\mbox{ in }L^p(\Omega).
\]
\end{Lemma}

\begin{proof}
For $u_h \in V^q_h(\Omega)$
and using either penalty term \eqref{eq:penalty} or \eqref{eq:penalty2} we first estimate the second term in the discrete energy \eqref{eq:energy_discrete_compact} as follows:
\begin{align*}
\int_{\Omega} R_h(u_h) : DW(\nabla_h u_h) & \leq \left(\int_{\Omega} |R_h(u_h)|^p \right)^\frac1p \left(\int_\Omega |DW(\nabla_h u_h)|^{\frac{p}{p-1}}\right)^{\frac{p-1}{p}} \\
& \leq \left(C_R \sum_{e \in E_h^i} h_e^{1-p} \int_e |\jumpop{u_h}|^p\right)^{\frac1p} \left(\int_\Omega C\left(1 + |\nabla_h u_h|^{p-1}\right)^{\frac{p}{p-1}}\right)^{\frac{p-1}{p}}\\
& \leq \left(C_R \sum_{e \in E_h^i} h_e^{1-p} \int_e |\jumpop{u_h}|^p\right)^{\frac1p} \left(  2^{\frac{1}{p-1}} C \int_\Omega 1 + |\nabla_h u_h|^{p}\right)^{\frac{p-1}{p}}
\end{align*}
by Lemma \ref{prop:Rh_bound} and the growth of $W$ which, combined with convexity, implies that
\begin{equation}\label{eq:DW}
|DW(\xi)| \lesssim 1 + |\xi|^{p-1}.
\end{equation}
The above inequality follows from \cite[Proposition 2.32]{dacorogna2007direct}.
Hence, by Young's inequality, we infer that 
\begin{align}\label{eq:compact_1}
\int_{\Omega} R_h(u_h) : DW(\nabla_h u_h) & \leq \frac{1}{p}C_R \delta^{1-p} \sum_{e \in E_h^i} h_e^{1-p} \int_e |\jumpop{u_h}|^p \nonumber \\
&\quad + \delta \frac{p-1}{p}  2^{\frac{1}{p-1}} C \int_\Omega \left(1 + |\nabla_h u_h|^{p}\right).
\end{align}
The coercivity assumption on $W$ says that
\[
W(\xi) \geq -C_0 + C_1 |\xi|^p.
\]
Thus, choosing $\delta^{-1} = \frac{2}{C_1} \frac{p-1}{p}  2^{\frac{1}{p-1}} C$, by \eqref{eq:compact_1} we find that
\[
\mathcal{E}_h(u_h) \geq \frac{C_1}2 \int_\Omega |\nabla_h u_h|^p - C - C(p,R)\sum_{e \in E_h^i} \frac{1}{h_e^{p-1}} \int_e |\jumpop{u_h}|^p + \alpha\,{\rm Pen}(u_h)
\] 
where
\[
C(p,R) = \frac1p C_R \delta^{1-p} = \frac{1}{p} C_R\left(\frac{2}{C_1}\frac{p-1}{p}  2^{\frac{1}{p-1}} C\right)^{p-1}.
\]
Next, note that since 
\begin{equation}\label{eq:norm_ineq}
 |u_h|^{p}_{W^{1,p}(\Omega,T_h)} \geq  \sum_{e\in E^i_h} \frac{1}{h_e^{p-1}}\int_e |\jumpop{u_h}|^p
\end{equation}
the penalty term in \eqref{eq:penalty} satisfies
\begin{align*}
{\rm Pen}(u_h) & = \left(1 + |u_h|^{p}_{W^{1,p}(\Omega,T_h)} \right)^{\frac{p-1}{p}} \left(\sum_{e\in E_h} \frac{1}{h_e^{p-1}}\int_e |\jumpop{u_h}|^p \right)^{\frac1p}\\
&\geq \left(1 + \sum_{e\in E^i_h} \frac{1}{h_e^{p-1}}\int_e |\jumpop{u_h}|^p \right)^{\frac{p-1}{p}} \left(\sum_{e\in E^i_h} \frac{1}{h_e^{p-1}}\int_e |\jumpop{u_h}|^p \right)^{\frac1p} \\
& \geq \sum_{e\in E^i_h} \frac{1}{h_e^{p-1}}\int_e |\jumpop{u_h}|^p,
\end{align*}
where the last inequality follows from the fact that the conjugate exponents $(p-1)/p$ and $1/p$ add up to $1$. Then, provided that $\alpha > C(p,R)$, we deduce that
\[
\mathcal{E}_h(u_h) \geq \frac{C_1}2 \int_\Omega |\nabla_h u_h|^p - C + (\alpha - C(p,R))\sum_{e\in E_h^i} \frac{1}{h_e^{p-1}} \int_e |\jumpop{u_h}|^p
\] 
and hence that
\[
|u_h|^p_{W^{1,p}(\Omega,T_h)} \lesssim 1 + \mathcal{E}_h(u_h).
\]
Furthermore, if $\mathcal{E}_h(u_h) \leq C$, $||u_h||_{L^p(\partial \Omega)}$ is uniformly bounded which implies from the Poincar\'{e} inequality~(\ref{eq:Poincare}) that $||u_h||_{L^p(\Omega)} \leq C^\prime$. 
From compact embeddings for broken Sobolev spaces, \cite[Theorem 6.1]{di2010discrete}, together with Lemma \ref{lemma:discrete_grad}, we infer Lemma \ref{lemma:compactness}. In particular, we find that
\[
u_h \to u \mbox{ in }L^p(\Omega)\mbox{ and }G_h(u_h) \rightharpoonup \nabla u\mbox{ in }L^p(\Omega),
\]
for some $u \in W^{1,p}(\Omega)$. To prove $u \in \mathbb{A}(\Omega)$ note that
\[
||u - u_0||_{L^p(\partial \Omega)} \le ||u_h - u_0||_{L^p(\partial \Omega)} + ||u - u_h||_{L^p(\partial \Omega)}.
\]
As $h\rightarrow 0$  the first term of the right hand side converges to zero due to the assumption that $\sup_h\mathcal{E}_h(u_h) \leq C$ and
$||u - u_h||_{L^p(\partial \Omega)} \rightarrow 0$  from the compact embedding in \cite[Lemma 8]{buffa2009compact}.
\end{proof}

\begin{Remark}
Regarding the penalty term \eqref{eq:penalty2}, provided that $p>2$ and using \eqref{eq:norm_ineq}, we find that 
\[
|u_h|^{p-2}_{W^{1,p}(\Omega,T_h)} =  \left(|u_h|^{p}_{W^{1,p}(\Omega,T_h)}\right)^{\frac{p-2}{p}} \geq \left(\sum_{e\in E^i_h} \frac{1}{h_e^{p-1}}\int_e |\jumpop{u_h}|^p \right)^{\frac{p-2}{p}}.
\]
Hence, the penalty term in \eqref{eq:penalty2} satisfies
\begin{align*}
{\rm Pen}(u_h) & = \left(1 + |u_h|^{p-2}_{W^{1,p}(\Omega,T_h)} \right) \left(\sum_{e\in E_h} \frac{1}{h_e^{p-1}}\int_e |\jumpop{u_h}|^p \right)^{\frac2p}  \\
&\geq \left(\sum_{e\in E_h} \frac{1}{h_e^{p-1}}\int_e |\jumpop{u_h}|^p \right)^{\frac{p-2}{p}}\left(\sum_{e\in E_h} \frac{1}{h_e^{p-1}}\int_e |\jumpop{u_h}|^p \right)^{\frac2p}\\
&= \sum_{e\in E_h} \frac{1}{h_e^{p-1}}\int_e |\jumpop{u_h}|^p
\end{align*}
due to the conjugacy of the exponents and Lemma \ref{lemma:compactness} also follows in this case.
\end{Remark}

We next prove the $\liminf$-inequality, providing a lower bound for the sequence $\mathcal{E}_h(u_h)$. 

\begin{Lemma}
\label{lemma:liminf}
Suppose that $W$ is convex, $u_h \in V^q_h(\Omega)$ and $u\in \mathbb{A}(\Omega)$ satisfy $u_h \to u$ in $L^p(\Omega)$. Then,
\[
\liminf_{h\to 0} \mathcal{E}_h(u_h) \geq \mathcal{E}(u).
\]
\end{Lemma}

\begin{proof}
Given $u_h$, $u$ as in the statement, suppose that $\liminf \mathcal{E}_h(u_h) \leq C$ as otherwise there is nothing to prove. By passing to a subsequence, we may assume that the limit 
is attained and that 
\[
\sup_h \mathcal{E}_h(u_h) \leq C.
\]
By Lemma \ref{lemma:compactness} we may also assume that $\|u_h\|_{W^{1,p}(\Omega,T_h)} \leq C$ and that
\begin{equation}\label{eq:liminf_1}
G_h(u_h) \rightharpoonup \nabla u\mbox{ in }L^p(\Omega).
\end{equation}
Note that since $W$ is convex and $G_h(u_h) \rightharpoonup \nabla u$ in $L^p(\Omega)$, it holds that
\begin{equation}\label{eq:liminf_2}
\liminf_{h\to 0}\int_\Omega W(G_h(u_h)) \geq \int_\Omega W(\nabla u) = \mathcal{E}(u).
\end{equation}
We thus add and subtract the term
\[
\int_\Omega W(G_h(u_h))
\]
in the discrete energy \eqref{eq:energy_discrete_compact} to find that
\[
\begin{aligned}
\mathcal{E}_h(u_h) & = \int_\Omega W(G_h(u_h)) + \int_\Omega W(\nabla_h u_h) - W(G_h(u_h)) \\
&\qquad + \int_\Omega R_h(u_h):DW(\nabla_h u_h) +  \alpha\, {\rm Pen}(u_h) \\
&\qquad =: I_h + II_h + III_h +  \alpha\, {\rm Pen}(u_h).
\end{aligned}
\]
By \eqref{eq:liminf_2} it suffices to show that
\begin{equation}\label{eq:sufficient}
\liminf_{h\to 0} \left[ II_h + III_h +  \alpha\, {\rm Pen}(u_h) \right] \geq 0. 
\end{equation}
We begin by estimating $III_h$. By H\"older's inequality we deduce that
\begin{equation*}
| III_h |  \leq  \left(\int_\Omega |R_h(u_h)|^p\right)^{\frac1p} \left(\int_\Omega | DW(\nabla_h u_h)|^{\frac{p}{p-1}}\right)^{\frac{p-1}{p}}.
\end{equation*}
By invoking the growth of $DW$ in \eqref{eq:DW} we infer that 
\begin{equation*}
\begin{aligned}
| III_h | & \leq  C \| R_h(u_h)\|_{L^p(\Omega)} \left(\int_\Omega | 1+ |\nabla_h u_h|^{p-1}|^{\frac{p}{p-1}}\right)^{\frac{p-1}{p}}\\
& \leq C \| R_h(u_h)\|_{L^p(\Omega)} \left(\int_\Omega \left( 1+ |\nabla_h u_h|^{p}\right)\right)^{\frac{p-1}{p}}\\
& = C  \| R_h(u_h)\|_{L^p(\Omega)} \left(1 + \|\nabla_h u_h\|_{L^p(\Omega)}^{p}\right)^{\frac{p-1}{p}}\\
& \leq C\,\left(C_R\right)^{\frac1p}  \left(\sum_{e\in E^i_h} \frac{1}{h_e^{p-1}}\int_e |\jumpop{u_h}|^p\right)^{\frac1p}\left(1 + \|\nabla_h u_h\|_{L^p(\Omega)}^{p}\right)^{\frac{p-1}{p}},
\end{aligned}
\end{equation*}
where the last inequality follows from Lemma \ref{prop:Rh_bound}. Hence, we obtain that
\begin{equation}\label{eq:III}
\begin{aligned}
| III_h | & \leq C\,\left(C_R\right)^{\frac1p}  \left(\sum_{e\in E^i_h} \frac{1}{h_e^{p-1}}\int_e |\jumpop{u_h}|^p\right)^{\frac1p} \left(1 + |u_h |_{W^{1,p}(\Omega,T_h)}^{p}\right)^{\frac{p-1}{p}} \\
& \leq  \,C\,\left(C_R\right)^{\frac1p} {\rm Pen}(u_h).
\end{aligned}
\end{equation}

Next, we estimate the term $II_h$. Note that the convexity and growth of $W$, see 
\cite[Proposition 2.32]{dacorogna2007direct}, imply the existence of a constant $C>0$ such that 
\begin{align}\label{eq:growthqc}
| W(\xi_1) - W(\xi_2) | \leq C \left( 1 + |\xi_1|^{p-1} +|\xi_2|^{p-1} \right) | \xi_1 - \xi_2 |.
\end{align}
Therefore, by H\"older's inequality, the relation $R_h(u_h) = \nabla_h u_h - G_h(u_h) $, and Lemma \ref{prop:Rh_bound} we find that 
\begin{equation*}
\begin{aligned}
| II_h | & \leq C \int_\Omega \left( 1 + |\nabla_h u_h|^{p-1} +|G_h(u_h)|^{p-1} \right) | \nabla_h u_h - G_h(u_h) | \\
& \leq C \left[\int_\Omega \left( 1 + |\nabla_h u_h|^{p-1} +|G_h(u_h)|^{p-1} \right)^{\frac{p}{p-1}} \right]^{\frac{p-1}{p}} \| R_h(u_h) \|_{L^p(\Omega)} \\
& \leq C\, C_R \left[\int_\Omega \left( 1 + |\nabla_h u_h|^{p} +|G_h(u_h)|^{p} \right) \right]^{\frac{p-1}{p}}  \left(\sum_{e\in E^i_h} \frac{1}{h_e^{p-1}}\int_e |\jumpop{u_h}|^p\right)^{\frac1p}.
\end{aligned}
\end{equation*}
By Lemma \ref{lemma:discrete_grad}, we deduce that 
\begin{equation}\label{eq:II}
\begin{aligned}
| II_h |  & \leq C\, C_R  \left( 1 + |u_h|_{W^{1,p}(\Omega,T_h)}^p  \right)^{\frac{p-1}{p}}  \left(\sum_{e\in E^i_h} \frac{1}{h_e^{p-1}}\int_e |\jumpop{u_h}|^p\right)^{\frac1p} \\
&\leq \,\,C\,C_R {\rm Pen}(u_h).
\end{aligned}
\end{equation}
Using \eqref{eq:III} and \eqref{eq:II}, we get that 
\[
II_h + III_h +  \alpha\, {\rm Pen}(u_h) \geq \left(\alpha - 2 C C_R\right) {\rm Pen}(u_h) 
\]
and choosing $\alpha > 2 C C_R $ we infer \eqref{eq:sufficient} which concludes the proof.
\end{proof}
\begin{Remark}
If instead we assume the penalty term \eqref{eq:penalty2}
we can prove Lemma~\ref{lemma:liminf} with the additional assumption that $W \in C^2$ and
\[
|D^2 W(\xi)| \lesssim 1 + |\xi|^{p-2}, \text{ } p\ge 2.
\]
Note this is a genuine assumption and not a consequence of convexity, see Appendix C in \cite{koumatos2020existence} for 
an example.
For the proof re-write the discrete energy as
\[
\begin{aligned}
\mathcal{E}_h(u_h) & \ge \int_\Omega W(G_h(u_h)) \\
&\quad + \int_\Omega W(\nabla_h u_h) - W(G_h(u_h)) - DW(G_h(u_h)) : R_h(u_h) \\
&\qquad + \int_\Omega \left(DW(G_h(u_h) - DW(\nabla_h u_h))\right) : R_h(u_h) +  \alpha\, {\rm Pen}(u_h)\\
&\qquad =: I_h + II_h + III_h +  \alpha\, {\rm Pen}(u_h).
\end{aligned}
\]
As before the convexity of $W$ implies that
\begin{equation}
\liminf_{h\to 0} I_h = \liminf_{h\to 0}\int_\Omega W(G_h(u_h)) \geq \int_\Omega W(\nabla u) = \mathcal{E}(u).
\end{equation}
Moreover, using the convexity of $W$ again, we have that
\[
W(\xi_1) - W(\xi_2) - DW(\xi_2):(\xi_1 - \xi_2) \geq 0,
\]
and recalling that $R_h(u_h) = \nabla_h u_h - G_h(u_h)$, we deduce that
\begin{equation}\label{eq:liminf_3}
II_h = \int_\Omega W(\nabla_h u_h) - W(G_h(u_h)) - DW(G_h(u_h)) : R_h(u_h) \geq 0.
\end{equation}
As for the term $\left(DW(G_h(u_h) - DW(\nabla_h u_h))\right) : R_h(u_h)$, we absorb it into the penalty term. Indeed, by assuming a $(p-2)$-growth on $D^2W$ , we infer that
\[
| DW(\xi_1) - DW(\xi_2) | \leq C \left( 1 + |\xi_1|^{p-2} + |\xi_2|^{p-2} \right)| \xi_1 - \xi_2 | 
\]
and, since $R_h(u_h) = \nabla_h u_h - G_h(u_h)$, that
\begin{align*}
|III_h| & \lesssim \int_\Omega \left(1 + |G_h(u_h)|^{p-2} + |\nabla_h u_h|^{p-2}\right)|R_h(u_h)|^2\\
& \lesssim \left(\int_\Omega \left(1 + |G_h(u_h)|^{p} + |\nabla_h u_h|^{p}\right)\right)^{\frac{p-2}{p}} \left(\int_\Omega |R_h(u_h)|^p\right)^{\frac2p}.
\end{align*}
However, by Lemma \ref{lemma:discrete_grad}, we find that
\[
\int_\Omega \left(1 + |G_h(u_h)|^{p} + |\nabla_h u_h|^{p}\right) \lesssim 1 +  |u_h|^p_{W^{1,p}(\Omega,T_h)}
\]
whereas, by Lemma \ref{prop:Rh_bound}, that
  \begin{align*}
   \int_\Omega |R_h(u_h)|^p  \leq 
  C_R \sum_{e \in E_h^i}h_e^{1-p} \int_e |\jumpop{u_h}|^p.
  \end{align*}
Thus, 
\begin{equation}\label{eq:liminf_4}
|III_h|  \lesssim \left(1 +  |u_h|^{p-2}_{W^{1,p}(\Omega,T_h)}\right)\left( \sum_{e }h_e^{1-p} \int_e |\jumpop{u_h}|^p\right)^{\frac2p} = {\rm Pen}(u_h).
\end{equation}
Combining \eqref{eq:liminf_2}--\eqref{eq:liminf_4}, we deduce that
\begin{align*}
\liminf_{h\to 0} \mathcal{E}_h(u_h) & \geq \mathcal{E}(u) + (\alpha - C){\rm Pen}(u_h) \geq \mathcal{E}(u),
\end{align*}
for $\alpha>0$ large enough.
\end{Remark}

The remaining ingredient in the proof of Theorem \ref{thm:main} is providing a recovery sequence which is given in the following lemma.

\begin{Lemma}
\label{lemma:limsup}
For every $u \in \mathbb{A}(\Omega)$ if  $q \ge 1$ and $h<1$ there exists $u_h \in V^q_h(\Omega)$  such that $u_h \to u$ in $L^p(\Omega)$ and 
\[
\limsup_{h\to 0} \mathcal{E}_h(u_h) \leq \mathcal{E}(u).
\]
\end{Lemma}

\begin{proof}
Here we show that for $u \in \mathbb{A}(\Omega)$ a recovery sequence $(u_h) \subset V_h^q$ is obtained where 
\begin{align}
|| u - u_h ||^p_{W^{1,p}(\Omega,T_h)} +  \sum_{e\in E^b_h}\frac{1}{h_e^{p-1}}\int_e |u_h - u_{0}|^p \rightarrow 0, 
\text{ as } h \rightarrow 0
\label{limsup_uh_conv}
\end{align}
and 
\begin{align}
\lim_{h\rightarrow 0} \mathcal{E}_h(u_h) = \mathcal{E}(u).
\label{eq:limsup_Eh_conv}
\end{align}
\\
First note that $C^\infty(\bar{\Omega})$ is dense in $W^{1,p}(\Omega)$ and one can create a sequence 
$(u_\delta) \subset C^\infty(\bar{\Omega})$ with the properties
\begin{align}
||u - u_\delta||_{W^{1,p}(\Omega)} \lesssim \delta \quad \text{and} \quad |u_\delta|_{W^{2,p}(\Omega)} \lesssim \frac{1}{\delta} |u|_{W^{1,p}(\Omega)}.
\label{eq:mol_ineq}
\end{align}
Choosing the sequence $u_{h, \delta} = I_h^q u_\delta$, where $I_h^q: W^{2,p}(\Omega) \rightarrow V_h^q(\Omega)$ is the standard nodal interpolation operation,
we recall the error estimates for all $q \ge 1$ 
\begin{align}
&|u_\delta - u_{h,\delta}|_{W^{i,p}(K)} \lesssim h_K^{2-i} |u_\delta |_{W^{2,p}(K)}, \text{ } i=0, 1, 2,
\label{eq:error_in_K} \\ 
&|u_\delta - u_{h,\delta}|_{L^p(e)} \lesssim h_K^{2-\frac{1}{p} }|u_\delta|_{W^{2,p}(K)}.
\label{eq:error_in_e} 
\end{align}
In the following, assuming that $h <1$, the above error estimates are combined with eq.~(\ref{eq:mol_ineq}) to get that
\begin{align}
&\sum_{K \in T_h} |u_\delta - u_{h, \delta}|_{W^{1,p}(K)}^p \lesssim h^p |u_\delta |_{W^{2,p}(\Omega)}^p \lesssim \frac{h^p}{\delta^p} |u|_{W^{1,p}(\Omega)}^p, 
\label{eq:bulk_terms}
\\
&\sum_{e \in E_h} \frac{1}{h_e^{p-1}} || \jumpop{u_\delta - u_{h, \delta}}||_{L^p(e)}^p \lesssim h^p \sum_{e \in E_h}|u_\delta|^p_{W^{2,p}(K_e)} 
\lesssim \frac{h^p}{\delta^p} |u|_{W^{1,p}(\Omega)}^p,
\label{eq:jump_terms}
\end{align}
here $K_e$ denotes the set of the elements containing $e \in E_h$. For the boundary terms applying 
Dirichlet boundary conditions through Nitsche's approach, 
the above estimates together with eq.~(\ref{eq:mol_ineq}) imply
\begin{equation}
\begin{aligned}
&\sum_{e\in E^b_h}\frac{1}{h_e^{p-1}}\int_e |u_{h,\delta} - u_0|^p  \lesssim \sum_{e\in E^b_h}\frac{1}{h_e^{p-1}}\int_e \left( |u_{h, \delta} - u_{\delta}|^p + |u_\delta - u|^p \right) 
\\
&\lesssim \left( \frac{h^{p}}{\delta^{p}} + \frac{\delta^p}{h^{p-1}}\right) || u||^p_{W^{1,p}(\Omega)}.
\label{eq:bnd_limsup}
\end{aligned}
\end{equation}
For $\beta \in (\frac{p-1}{p}, 1)$, choosing $\delta = h^\beta$ equations~(\ref{eq:bulk_terms}), (\ref{eq:jump_terms}) and (\ref{eq:bnd_limsup})
tend to $0$ as $h\rightarrow 0$. Here $u \in W^{1,p}(\Omega)$  ($u_\delta$ is continuous by construction), which means $\jumpop{u} =0$ a.e. on every $e \in E_h^i$
 and as $h\rightarrow 0$
\begin{align}
    \norm{u -u_{h,\delta}}^p_{W^{1,p}(\Omega, T_h)} \le \norm{u -u_\delta}_{W^{1,p}(\Omega)}
    + \norm{u_\delta - u_{h,\delta}}_{W^{1,p}(\Omega, T_h)}^p \rightarrow 0.
\end{align}
Therefore, we found a sequence in $V_h^q(\Omega)$ such that eq.~\eqref{limsup_uh_conv} holds.
Next for the convergence of the discrete functional, 
eq.~(\ref{eq:limsup_Eh_conv}), note that 
\begin{equation}\label{eq:limsup_1}
G_h(u_h) \to \nabla u \mbox{ in }L^p(\Omega). 
\end{equation}
Indeed, recalling that $G_h(u_h) = \nabla_h u_h - R_h(u_h)$, $\jumpop{u} = 0$ a.e. and using Lemma \ref{prop:Rh_bound}, we find that
\begin{align*}
\int_\Omega |G_h(u_h) - \nabla u|^p & \lesssim \int_\Omega |\nabla_h u_h - \nabla u|^p + \int_\Omega |R_h(u_h)|^p\\
& \lesssim  \int_\Omega |\nabla_h u_h - \nabla u|^p + \sum_{e\in E_h^i} \frac{1}{h_e^{p-1}} \int_e |\jumpop{u_h} |^p\\
& =  \sum_K \int_K |\nabla_h u_h - \nabla u|^p + \sum_{e\in E_h^i} \frac{1}{h_e^{p-1}} \int_e |\jumpop{u_h - u} |^p\\
& = |u_h - u|^p_{W^{1,p}(\Omega,T_h)}\to 0.
\end{align*}
As in Lemma \ref{lemma:liminf}, we can write the discrete energy as
\[
\begin{aligned}
\mathcal{E}_h(u_h) & = \int_\Omega W(G_h(u_h)) + \int_\Omega W(\nabla_h u_h) - W(G_h(u_h)) \\
&\qquad + \int_\Omega R_h(u_h):DW(\nabla_h u_h) +  \alpha\, {\rm Pen}(u_h)  \\
&\qquad =: I_h + II_h + III_h +  \alpha\, {\rm Pen}(u_h).
\end{aligned}
\]
Vitali's convergence theorem, \eqref{eq:limsup_1}, and the growth of $W$ say that
\begin{equation}\label{eq:limsup_2}
I_h = \int_\Omega W(G_h(u_h)) \to \int_\Omega W(\nabla u) = \mathcal{E}(u) \mbox{ as } h\to0
\end{equation}
and since the boundary terms converge to zero, eq.~(\ref{limsup_uh_conv}),  it remains to show that $II_h$, $III_h$ 
as well as ${\rm Pen}(u_h)$ all vanish in the limit $h\to 0$. Using the estimates \eqref{eq:III} and \eqref{eq:II}, we find that 
\[
 II_h + III_h \lesssim  {\rm Pen}(u_h) 
\]
and it thus suffices to prove that 
\[
\limsup_{h\to0}{\rm Pen}(u_h)  = 0.
\]
Indeed, since $ |u_h - u |_{W^{1,p}(\Omega,T_h)} \to 0$, it follows that
\[
|u_h|_{W^{1,p}(\Omega,T_h)} \leq C
\]
and hence
\begin{align*}
{\rm Pen}(u_h) & \lesssim \left( \sum_{e \in E_h} \frac{1}{h_e^{p-1}} \int_e |\jumpop{u_h} |^p\right)^{\frac1p}  
= \left( \sum_{e\in E_h} \frac{1}{h_e^{p-1}} \int_e |\jumpop{u_h - u} |^p\right)^{\frac1p} \\
 & \lesssim \left( |u_h -u|_{W^{1,p}(\Omega,T_h) } + 
 \sum_{e\in E_h^b} \frac{1}{h_e^{p-1}} \int_e |\jumpop{u_h - u} |^p \right)^{\frac1p} \to 0 .
\end{align*}
This completes the proof.
\end{proof}

\begin{Remark}
If instead we assume the penalty term \eqref{eq:penalty2}, we re-write the discrete energy as
\[
\begin{aligned}
\mathcal{E}_h(u_h) & = \int_\Omega W(G_h(u_h)) \\
&\quad + \int_\Omega W(\nabla_h u_h) - W(G_h(u_h)) - DW(G_h(u_h)) : R_h(u_h) \\
&\qquad + \int_\Omega \left(DW(G_h(u_h) - DW(\nabla_h u_h))\right) : R_h(u_h) +  \alpha\, {\rm Pen}(u_h)\\
&\qquad =: I_h + II_h + III_h +  \alpha\, {\rm Pen}(u_h).
\end{aligned}
\]
As before
\begin{equation}
I_h = \int_\Omega W(G_h(u_h)) \to \int_\Omega W(\nabla u) = \mathcal{E}(u)
\end{equation}
and it remains to show that $II_h$, $III_h$, and ${\rm Pen}(u_h)$ vanish in the limit $h\to 0$. For $II_h$, the $(p-2)$-growth of $D^2W$ implies that
\begin{align*}
| II_h |  & \leq \int_\Omega \int_0^1(1-t) D^2W(G_h(u_h) + t R_h(u_h)) R_h(u_h) : R_h(u_h)\,dt \\
& \lesssim \int_\Omega \left(1 + |G_h(u_h)|^{p-2} + |R_h(u_h)|^{p-2}\right)|R_h(u_h)|^2 \\
& \lesssim \left(\int_\Omega \left(1 + |G_h(u_h)|^{p} + |R_h(u_h)|^{p}\right)\right)^{\frac{p-2}{p}} \left(\int_\Omega |R_h(u_h)|^p\right)^{\frac2p}
\end{align*}
As in the proof of Lemma \ref{lemma:liminf}, by Lemma \ref{lemma:discrete_grad} and Lemma \ref{prop:Rh_bound}, we find that
\[
\int_\Omega \left(1 + |G_h(u_h)|^{p} + |R_h(u_h)|^{p}\right) \lesssim 1 +  |u_h|^p_{W^{1,p}(\Omega,T_h)}.
\]
Then, since by Lemma \ref{prop:Rh_bound},
  \begin{align*}
   \int_\Omega |R_h(u_h)|^p  \leq 
  C_R \sum_{e \in E_h^i}h_e^{1-p} \int_e |\jumpop{u_h}|^p,
  \end{align*}
we deduce that
\begin{align*}
|II_h| \lesssim \left(1 +  |u_h|^{p-2}_{W^{1,p}(\Omega,T_h)}\right) \left(\sum_{ei}h_e^{1-p} \int_e |\jumpop{u_h}|^p\right)^{\frac2p} \lesssim {\rm Pen}(u_h).
\end{align*}
For the term $III_h$, exactly as we did in the proof of Lemma \ref{lemma:liminf}, see \eqref{eq:liminf_4}, we have that
\begin{equation*}
|III_h|  \lesssim \left(1 +  |u_h|^{p-2}_{W^{1,p}(\Omega,T_h)}\right)\left( \sum_{e }h_e^{1-p} \int_e |\jumpop{u_h}|^p\right)^{\frac2p} \lesssim {\rm Pen}(u_h).
\end{equation*}
We are thus left to prove that ${\rm Pen}(u_h)\to 0$, as $h\to 0$. Indeed, since $| u_h - u |_{W^{1,p}(\Omega,T_h)} \to 0$, and exactly as we did before, it follows that
\[
|u_h|_{W^{1,p}(\Omega,T_h)} \leq C
\]
and that
\begin{align*}
{\rm Pen}(u_h) & \lesssim \left( \sum_e \frac{1}{h^{p-1}_e} \int_e |\jumpop{u_h} |^p\right)^{\frac2p}  \\
& \lesssim \left( \sum_e \frac{1}{h_e^{p-1}} \int_e |\jumpop{u_h - u} |^p\right)^{\frac2p} \lesssim |u_h -u|^{\frac2p}_{W^{1,p}(\Omega,T_h)}  \to 0,
\end{align*}
completing the proof of Lemma \ref{lemma:limsup} for the penalty \eqref{eq:penalty2}.
\end{Remark}

We conclude by demonstrating the proof of Theorem \ref{thm:main}.

\begin{proof}[Proof of Theorem \ref{thm:main}]

Suppose that $u_h \in V^q_h$ is given such that
\[
\mathcal{E}_h(u_h) = \inf_{V^q_h} \mathcal{E}_h.
\]
Note that for any $v \in \mathbb{A}(\Omega)$, by Lemma \ref{lemma:limsup} we can find a sequence $v_h \in V^q_h$ such that
\[
v_h \to v \mbox{ in }L^p(\Omega)
\]
and $\limsup \mathcal{E}_h(v_h) \leq \mathcal{E}(v)$. In particular, up to a subsequence,
\[
\mathcal{E}_h(u_h) \leq \mathcal{E}_h(v_h) \leq \limsup \mathcal{E}_h(v_h) \leq  \mathcal{E}(v) < \infty.
\]
Hence, by Lemma \ref{lemma:compactness}, there exists $u\in \mathbb{A}(\Omega)$ such that
\[
u_h \to u \mbox{ in }L^p(\Omega)
\]
and, for any $v \in \mathbb{A}(\Omega)$,
\[
\mathcal{E}(u) \leq \liminf_{h\to0} \mathcal{E}_h(u_h) \leq \limsup \mathcal{E}_h(v_h) \leq \mathcal{E}(v),
\]
where $v_h$ is the recovery sequence for $v$ whose existence is guaranteed by Lemma \ref{lemma:limsup}. That is, 
\[
\mathcal{E}(u) = \inf_{\mathbb{A}(\Omega)} \mathcal{E}
\]
completing the proof.
\end{proof}

\section{A priori estimates} In this section we shall focus on a priori error estimates between the exact solution $u$ and the solution $u_h$ of the discrete  \emph{Euler-Lagrange equations}. 
Under certain assumptions on $u$ and on the energy $W$ we will show that $u_h$ exists and satisfies optimal-order a priori error bounds. To this end, we introduce some further
notation.  As pointed out in \eqref {D_en_dg_0} the 
  first variation  of \eqref{en_dg_1} is 
\begin{equation}\label{D_en_dg_0_n0}
\begin{aligned}
a_h(u_h, v_h) :=&\< D\mathcal{E}_{dg} [u_h] , v_h \>  = \sum_{K \in T_h}\int_{K}  S(\nabla u_h(x)) \nabla v_h(x)\\ 
 & -\sum_{e \in E^i_h}\int_{e}\ S  (\Av{ \nabla u_h}  )  \cdot \J {v_h \otimes n_e}  ds
 \\
 &-\sum_{e \in E^i_h}\int_{e}\ 
\Bigl (   \partial _\zeta S (\Av{\nabla u_h  } )\, \J {u_h  \otimes n_e }  {\Bigl) \cdot}
   \Av{\nabla v_h }   
   ds \\
&   + \< D \text{Pen}\, ( {u_h} ), v_h \>.
 \end{aligned}
\end{equation}
Notice that critical points $u_h$ of the Euler-Lagrange equations satisfy
\begin{equation}\label{D_en_dg_est_u_h}
\begin{aligned}
a_h(u_h, v_h) =  \< f , v_h \> \qquad \text {for all }  v_h \in  V^q_h(\Omega)\, .
 \end{aligned}
\end{equation}

%

The minimiser of the continuous problem will satisfy   (employing the summation convention), 
\begin{equation}
\begin{aligned} \label{D_en_dg_est_u}
 a (u , v ) :=  
          \int_\Omega  S_{i\alpha} (\nabla y(x))   \,  \partial_\alpha  v^i  (x) \, dx  = \< f , v  \> \qquad \text {for all }     v\in V \, .
\end{aligned}
\end{equation}
We shall need the derivative of  the Piola-Kirchhoff stress  tensor $S = \partial _\zeta \, W = \{ A_{i\alpha j \beta} \}$ defined by 
\begin{equation}
\begin{aligned}
A_{i\alpha j \beta}  (\zeta)  &=  \frac{\partial   } {\partial  {\zeta_{j\beta}}} S_{i \alpha} (\zeta)    =   \frac{\partial ^2 } {\partial  {\zeta_{i\alpha}} \partial  {\zeta_{j\beta}}}  W (\zeta )  \, .
\end{aligned}
\end{equation}
Dealing with smooth solutions of the dynamic and the static problem,\\   \cite{DafHrusa_1985, makr_1993, ortner2007discontinuous}, a  coercivity assumption on $A_{i\alpha j \beta}  $ was instrumental: 
\begin{equation}\label{A_coerc_1}
\sum _{i, \alpha , j , \beta }A_{i\alpha j \beta}  (\eta ) \zeta _\alpha \zeta _\beta \xi _i \xi _j  \geq M_1 | \zeta | ^2 |\xi | ^2   \, ,   \qquad \text {for  } M_1  >0\, , 
\end{equation}
and  for all  
$\zeta  ,  \xi \in \R ^d\, .$  In addition $A_{i\alpha j \beta} $ is assumed to be symmetric. Hypothesis  \eqref{A_coerc_1} allows the inclusion of 
physically relevant energies, \cite{DafHrusa_1985}, and it will be assumed in the sequel.   It implies a G\aa rding type inequality
\begin{equation}\label{A_coerc}
\sum _{i, \alpha , j , \beta } \int_\Omega A_{i\alpha j \beta}  (\nabla u(x))   \,  \partial_\beta  v^j  (x) \partial_\alpha  v^i  (x) \, dx    \geq M_0 (u)  \| v \| _ {H^1 (\Omega)} - \mu _0 (u) \| v \| _ {L^2 (\Omega)} \, ,
\end{equation}
which is the key assumption in our analysis. 

\subsection{Error Analysis}

In this section we perform the convergence analysis of our discontinuous Galerkin finite scheme for the penalty term \eqref{eq:penalty2} whenever $p=2$. The general case $p>2$ 
depends  on generalised G\aa rding inequalities, requires more delicate analysis, and is beyond the scope of this paper; see Remark \ref{gen_Garding}.

As mentioned, we follow the general plan introduced in  \cite{makr_1993, ortner2007discontinuous} considering appropriate linearisation problems and fixed point arguments. Note that, in comparison to form \eqref{D_en_dg_0_n0}, the method considered in \cite{ortner2007discontinuous} was similar but did not include the term involving $ \partial _\zeta S.$ As a result, the additional technical aspects in the analysis are minor, as discussed below, and can be handled using standard estimates.
 In fact, one can show by following the arguments of \cite{ortner2007discontinuous} that the linearised form satisfies a discrete  G\aa rding-type inequality, see  \cite[Lemma 4.1]{ortner2007discontinuous}. We build on this result, (which implies \eqref{d_garding}), and we show 
the discrete inf-sup condition Lemma 4.1 without additional assumptions. \cite[Lemma 4.2, Lemma 4.3]{ortner2007discontinuous} are used as well. Lemma 4.1. provides   what is needed to complete the fixed point argument, see below for details. The result of \cite{ortner2007discontinuous} in the stationary case assumes a stronger coercivity,  (\eqref{A_coerc} with $m_0 =0$),  see \cite[Section 5, p.1382]{ortner2007discontinuous}. We briefly describe the main steps of the analysis and we provide details of the proof of the Lemma 4.1 below. 

\begin{Remark}\label {gen_Garding} The starting point of the $L^2-$ analysis considered herein, as well as in the previous works, \cite{ortner2007discontinuous}, is the G\aa rding inequality \eqref{A_coerc}. Generalising \eqref{A_coerc} to $L^p$ setting is not straightforward, see e.g. \cite{Koumatos_Vikelis_2021} and its references. The form of the  G\aa rding inequality might depend on quantities  
of the form $\| v \| _{L^p (\Omega ) }^p + \| v \| _{L^2 (\Omega )} ^2.$ 
It is possible, although technically more demanding, to extend our analysis under the validity of such assumptions. This will be the subject of a forthcoming work.  
\end{Remark}

Let $W_h= I _h u $ be an appropriate conforming finite element projection  of $u$ (which is assumed to be continuous and satisfies typical optimal order error bounds). 
Then, the consistency error is defined as 
\begin{equation}\label{cons_error}
\begin{aligned}
a_h(W_h, v_h)  = \< f, v_h \> + \< R_h, v_h\>
 \end{aligned}
\end{equation}
where $\< R_h, v_h\>$ is the consistency error satisfying optimal order error bounds. 
Then, 
\begin{equation}\label{error_eq_1}
\begin{aligned}
a_h(u_h, v_h)  - a_h(W_h, v_h) =   - \< R_h, v_h\>\, . 
 \end{aligned}
\end{equation}
Next we shall evaluate the left hand side of this equation:
\begin{equation*}\label{D_en_dg_0_nt}
\begin{aligned}
a_h(W_h, &v_h)  - a_h(u_h, v_h) =  \sum_{K \in T_h}\int_{K}  (S(\nabla W_h(x)) - S(\nabla u_h(x)) )\nabla v_h(x)\\ 
 & -\sum_{e \in E^i_h}\int_{e}\Big (\ S  (\Av{ \nabla W_h}  ) - \ S  (\Av{ \nabla u_h}  ) \Big) \cdot \J {v_h \otimes n_e}  ds
 \\
 &-\sum_{e \in E^i_h}\int_{e}\ 
\Bigl (  ( \partial _\zeta S (\Av{\nabla W_h  }  -  \partial _\zeta S (\Av{\nabla u_h  } )\, \J {W_h  \otimes n_e }  {\Bigl) \cdot}
   \Av{\nabla v_h }   
   ds \\
   &-\sum_{e \in E^i_h}\int_{e}\ 
\Bigl (   \partial _\zeta S (\Av{\nabla u_h  } )\, \J {(W_h - u_h ) \otimes n_e }  {\Bigl) \cdot}
   \Av{\nabla v_h }   
   ds \\
&   + \< D \text{Pen}\, ( {W_h - u_h} ), v_h \>.
 \end{aligned}
\end{equation*}
Since $W_h $ is continuous, the third term on the right hand side vanishes. Thus,  we conclude, 
\begin{equation}\label{error_eq_2_0}
\begin{aligned}
A_h(u_h;  u_h - W_h , v_h)  =   - \< R_h, v_h\>\, , 
 \end{aligned}
\end{equation}
where we have set, 
\begin{equation}\label{error_eq_2}
\begin{aligned}
\overline A (\nabla w \ ; \ \nabla u   )  =    \int _0^1    \partial _\zeta S (\nabla w  + \tau (\nabla u  -\nabla w  ) ) d\tau \, , 
 \end{aligned}
\end{equation}
and 
\begin{equation*}\label{D_en_dg_0_nt2}
\begin{aligned}
A_h(\varphi\  ;\ &  w_h , v_h)    \\
:=& \sum_{K \in T_h}\int_{K}  \overline A (\nabla W_h \ ; \ \nabla \varphi   )  \nabla w_h(x) \nabla v_h(x)\\ 
 & -\sum_{e \in E^i_h}\int_{e}\Big ( \overline A (\Av{\nabla W_h}  \ ; \ \Av{\nabla \varphi}   )   \Av{ \nabla (w_h )}  )  \Big) \cdot \J {v_h \otimes n_e}  ds
 \\
    &-\sum_{e \in E^i_h}\int_{e}\ 
\Bigl (   \partial _\zeta S (\Av{\nabla \varphi  } )\, \J {w_h   \otimes n_e }  {\Bigl) \cdot}
   \Av{\nabla v_h }   
   ds \\
&   + \< D \text{Pen}\, ( {w_h} ), v_h \>.
 \end{aligned}
\end{equation*}
Therefore for fixed $\varphi ,$   $A_h(\varphi\  ;\   w_h , v_h) $  is a bilinear form (with respect to $w_h , v_h$). This form is closely related to the standard symmetric interior penalty DG 
bilinear form. To derive the estimates one needs to define a set 
\begin{equation}\label{Z}
\begin{aligned}
\mathbf{Z} _\delta (u)= \{ \Phi \in L ^ \infty ( \R ^{d\times d}) \ :  \|  \Phi - \nabla u \| _{ L ^\infty} < \delta \,  \} \, ,
 \end{aligned}
\end{equation}
where $u$ is the (assumed smooth) minimiser of the continuous problem  and $\delta$ is a fixed number which guarantees that $\Phi \in \mathbf{Z} _\delta (u)$ satisfy $ \Phi (x) \in \mathcal {O}. $
Without being very precise, if a function $\phi \in V^q_h(\Omega) $ and its piecewise derivative is in $\mathbf{Z}_\delta (u)  $ then the coercivity of the principal part of $A_h(\varphi\  ;\  w_h , v_h)$
will follow in view of \eqref{A_coerc} and of the fact $\nabla W_h $ will be close to $\nabla u$ for $h$ small enough.  To simplify matters we shall assume that 
\begin{equation}\label{Z_W_h}
\begin{aligned}
\text{for $h< h_0,$ if
 $\Phi \in \mathbf{Z} _\delta (W_h)$ then  $ \Phi (x) \in \mathcal {O}. $}
 \end{aligned}
\end{equation}
Next, 
define the set $\JJ$ as 
\begin{equation}\label{J}
\begin{aligned}
\JJ = \{  v_h \in V^q_h(\Omega) \ : \  & \| v_ h - W_h \| _{H ^1 (\Omega,T_h) } \leq C_\star (u ) h^r \\
&\quad \text{ and in piecewise sense } 
 \nabla v_h \in \mathbf{Z} _\delta (W_h) \, \}\, , 
 \end{aligned}
\end{equation}
where we use the notation $H ^1 (\Omega,T_h) = {W^{1,2}(\Omega,T_h)}\, .$

The next result is essential. The assumption that $\Omega $ is convex can be relaxed: 
\begin{Lemma}
\label{lemma:infsup} 
Assume that $\Omega $ is convex and that $\varphi \in \JJ, $ then there exists a $h_1>0$ such that for all $ h\leq h_1,$ there exist a constant $ \gamma _0$ such that 
\begin{equation}\label{inf_sup}
\begin{aligned}
 \sup _{ v_h \in V^q_h(\Omega) } \frac { A_h(\varphi\  ;\    w_h , v_h) } { \| v_ h   \| _{H ^1 (\Omega,T_h) } } \geq \gamma _0\, \| w_ h   \| _{H ^1 (\Omega,T_h) }
  \end{aligned}
\end{equation}
where the constant $ \gamma _0$ is independent of  $w_h, \varphi \in V^q_h(\Omega)$ and of $h.$ 
\end{Lemma}
\begin{proof}
%
%
%
%
%
We adopt in our setting an original idea of \cite{Schatz_1974}, and its \emph{infsup} version from \cite{MakrIB_1996}.
We first consider the auxiliary bilinear forms
\begin{equation}\label{aux_A_h}
\begin{aligned}
D_h (  w_h , v_h)  &=  \sum_{K \in T_h}\int_{K}   \partial _\zeta S ( \nabla u  )\nabla w_h(x) \nabla v_h(x)\\ 
 & -\sum_{e \in E^i_h}\int_{e}\Big (  \partial _\zeta S ( \nabla u  ) )   \Av{ \nabla (w_h )}  )  \Big) \cdot \J {v_h \otimes n_e}  ds
 \\
    &-\sum_{e \in E^i_h}\int_{e}\ 
\Bigl (   \partial _\zeta S ( \nabla u  )\, \J {w_h   \otimes n_e }  {\Bigl) \cdot}
   \Av{\nabla v_h }   
   ds \\
&   + \< D \text{Pen}\, ( {w_h} ), v_h \> + \mu _1 \< {w_h} , v_h \>,   \quad \text{and}\\
D  (  w  , v )  &=  \sum_{K \in T_h}\int_{K}   \partial _\zeta S ( \nabla u  )  \nabla w (x) \nabla v (x) + \mu _1 \<   w  , v  \> \\ 
			&= :  A ( u\ ; \    w  , v  )  + \mu _1 \<   w  , v  \>
 \end{aligned}
\end{equation} where $\mu _1$ is chosen such that both forms are coercive. Define now the projection $ P_D : H^1_0(\Omega) \to V^q_{h, C}(\Omega) \cap H^1_0 (\Omega)$ by  
\begin{equation*}
\begin{aligned}
D  (  \psi  , P_D v)  &= D  (  \psi  ,   v), \quad \text{for all } \psi \in V^q_{h, C}(\Omega) \cap H^1_0 (\Omega)\, .
 \end{aligned}
\end{equation*} 
Let  $w_h$ be a given element of $V^q_h(\Omega) $. 
Let  $v_h  =  w_h + P_D z $ where $z$  shall be defined below. 
Then, denoting by $ \overline w_h$ the reconstruction operator of $w_h$ defined in \cite{karakashian2003posteriori}, we have
\begin{equation*}\label{D_en_dg_0_nt3}
\begin{aligned}
A_h &(\varphi\  ;\    w_h , v_h)    =  A_h(\varphi\  ;\    w_h , w_h)  + A_h(\varphi\  ;\    w_h , P_D z)  \\
 &=  A_h(\varphi\  ;\    w_h , w_h)    + \mu _1 \<   w_h , w_h \>  + A_h(\varphi\  ;\    w_h -\overline w_h, P_D z) \\
 &  + A_h(\varphi\  ;\     \overline w_h, P_D z)   -  A ( u\ ; \   \overline w_h , P_D z )  +  A ( u\ ; \   \overline w_h , P_D z )
 - \mu _1 \<   w_h , w_h \> \\
 &=  A_h(\varphi\  ;\    w_h , w_h)    + \mu _1 \<   w_h , w_h \>  + A_h(\varphi\  ;\    w_h -\overline w_h, P_D z) \\
 &  + \Big [ A_h(\varphi\  ;\     \overline w_h, P_D z)   -  A ( u\ ; \   \overline w_h , P_D z ) \Big ] +  A ( u\ ; \   \overline w_h ,  z )
 - \mu _1 \<   w_h , w_h \> \\
  & \qquad   + \mu _1 \<  \overline w_h , z- P_D z \> \\
  &= \left[ A_h(\varphi\  ;\    w_h , w_h)    + \mu _1 \<   w_h , w_h \>  + A_h(\varphi\  ;\    w_h -\overline w_h, P_D z)\right ] \\
 &  + \Big [ A_h(\varphi\  ;\     \overline w_h, P_D z)   -  A ( u\ ; \   \overline w_h , P_D z ) \Big ] +  \Big [   A ( u\ ; \   \overline w_h ,  z ) - \mu_1 \<  \overline    w_h , w_h \> \Big ]\\
&  - \mu _1 \<   w_h - \overline    w_h , w_h \>  
     + \mu _1 \<  \overline w_h , z- P_D z \> \, =: T_1+T_2+T_3+T_4+T_5.\\
 \end{aligned}
\end{equation*}
Select now $z\in H^1_0 (\Omega) $ as the solution of 
\begin{equation}\label{def_z}
\begin{aligned}
 A ( u\ ; \   v ,  z ) = \mu_1 \<       v , w_h \>  \qquad \text{for all } v\in  H^1_0 (\Omega)\,,
 \end{aligned}
\end{equation}
and then  $T_3=0$. We claim that
\begin{align}\label{eq:T1}
    A_h (\varphi\  ;\    w_h , v_h) \geq  C_0  \|  w_h \| _{H ^1 (\Omega,T_h) }^2.
\end{align}
Indeed, following the arguments of   \cite[Lemma 4.1]{ortner2007discontinuous}   we can show that
\begin{equation}
    \label{d_garding}
    \begin{split}
         A_h(\varphi\  ;\    w_h , w_h)\gtrsim (1-\varepsilon^2)\sum_K&\int_K |\nabla w_h|^2-\frac{1}{\varepsilon^2}\sum_{e\in E^i_h}\frac{1}{h_e}\int_{e}|\jumpop{w_h}|^2\\
    &+\alpha \sum_{e\in E^i_h}\frac{1}{h_e}\int_{e}|\jumpop{w_h}|^2-\|w_h\|_{L^2(\Omega)}^2.
    \end{split}
\end{equation}
Additionally, (compare to \cite[Lemma 4.2]{ortner2007discontinuous})   we have 
\begin{align*}
     A_h(\varphi\  ;\    w_h -\overline w_h, P_D z)&\lesssim \|  w_h -\overline{w_h}\| _{H ^1 (\Omega,T_h) }\|  P_Dz \| _{H ^1 (\Omega) }  \\
     &\lesssim \|  w_h -\overline{w_h}\| _{H ^1 (\Omega,T_h) }\|  w_h \| _{H ^1 (\Omega,T_h) }\\
     &\lesssim \varepsilon^2 \|  w_h \|^2 _{H ^1 (\Omega,T_h) }+\frac{1}{\varepsilon^2}\|  w_h -\overline{w_h}\|^2 _{H ^1 (\Omega,T_h) }\\
     &\lesssim \varepsilon^2 \|  w_h \|^2 _{H ^1 (\Omega,T_h) }+\frac{1}{\varepsilon^2}\sum_{e\in E^i_h}\frac{1}{h_e}\int_{e}|\jumpop{w_h}|^2.
\end{align*}
Next, concerning the second term $T_2$, we use  \cite[Lemma 4.3]{ortner2007discontinuous} to get that
\begin{align*}
    -T_2&\lesssim \|\nabla\varphi-\nabla u\|_\infty\|\overline{w_h}\| _{H ^1 (\Omega ) }\|  P_Dz \| _{H ^1 (\Omega) }\lesssim \delta  \|  w_h \|^2 _{H ^1 (\Omega,T_h) }.
\end{align*}
For the last two terms, $T_4$ and $T_5,$ we have that
\begin{align*}
    \big |\mu_1\langle w_h-\overline{w_h},w_h\rangle\big|&\lesssim  \mu_1^2\|w_h-\overline{w_h}\|^2_{L^2(\Omega)}+\|w_h\|_{L^2(\Omega)}^2\\
    &\lesssim h^2\mu_1^2\sum_{e\in E^i_h}\frac{1}{h_e}\int_{e}|\jumpop{w_h}|^2+\|w_h\|_{L^2(\Omega)}^2,
\end{align*}
and
\begin{align*}
    \big |\mu _1 \<  \overline w_h ,& z- P_D z \> \big |\lesssim \|\overline{w_h}\|^2_{L^2(\Omega)}+\mu_1^2\|z-P_Dz\|_{L^2(\Omega)}^2\\
    &\lesssim h^2\sum_{e\in E^i_h}\frac{1}{h_e}\int_{e}|\jumpop{w_h}|^2+h\|w_h\|_{L^2(\Omega)}^2+h^{2}\mu_1^2 \|  w_h \| _{H ^1 (\Omega,T_h) }^2\, .
\end{align*}

Combining all the above we finally conclude that
\begin{align*}
   A_h (\varphi\  ;\    w_h , v_h)&\gtrsim (1-\varepsilon^2)\sum_K\int_K |\nabla w_h|^2+(\alpha-\frac{1}{\varepsilon^2}-h^2\mu_1^2-h)\sum_{e\in E^i_h}\frac{1}{h_e}\int_{e}|\jumpop{w_h}|^2\\
   &+(\mu_1-1-h)\|w_h\|_{L^2(\Omega)}^2-(\delta+h^2\mu_1^2+\varepsilon^2)\|  w_h \| _{H ^1 (\Omega,T_h) }^2.
\end{align*}
Choose now for example $\mu_1>1,\,\varepsilon<1$ and $\alpha>1/\varepsilon^2$ and then
\begin{align*}
     A_h (\varphi\  ;\    w_h , v_h)\gtrsim (1-\delta-h^2\mu_1^2-\varepsilon^2)\|  w_h \| _{H ^1 (\Omega,T_h) }^2,
\end{align*}
and by choosing $\delta<1/2-\varepsilon^2$ (for a possibly smaller $\varepsilon$) and $h$ small enough we prove our claim.
\noindent We recall here that by using standard properties of conforming finite elements and the elliptic regularity of \eqref{def_z} we know that
\begin{equation*}\label{D_en_dg_0_n}
\begin{aligned}
\| z- P_D z \| _{L^2 (\Omega,T_h) }   \leq C h ^{2} \| z\| _{H ^2 (\Omega) }\leq C h ^{2} \| w _h\|  _{L^2  (\Omega)}  \leq C h ^{2} \| w _h\|  _{H ^1 (\Omega,T_h) }   \, .\\
 \end{aligned}
\end{equation*}
Furthermore, we have, 
\begin{equation*}\label{D_en_dg_0_nt4}
\begin{aligned}
\|  v_h \| _{H ^1 (\Omega,T_h) }  \leq      \|     w_h  \| _{H ^1 (\Omega,T_h) } + \| P_D z \| _{H ^1 (\Omega) }    \leq  \|     w_h  \| _{H ^1 (\Omega,T_h) } + \|   z \| _{H ^1 (\Omega) }  \, \\
 \leq c_2   \|     w_h  \| _{H ^1 (\Omega,T_h) }.
 \end{aligned}
\end{equation*}
Combining the above bounds we complete the proof by assuming that   $ h\leq h_1,$ for an appropriate   $h_1>0.$ 
\end{proof}
Next, following  \cite{makr_1993, ortner2007discontinuous}, we show that the map $\N : \JJ \to V^q_h(\Omega) $ defined by 
\begin{equation}\label{error_eq_3}
\begin{aligned}
A_h(\phi ;  \N (\phi)   - W_h , v_h)  =  A_h(u ;  u - W_h , v_h)   , 
 \end{aligned}
\end{equation}
has a unique fixed point in $\JJ.$ To this aim, we first prove that $\N$ maps $\JJ$ into itself. For $v_h\in V^q_h(\Omega)$, by continuity of the bilinear form, we get that
\begin{align*}
    \frac { A_h(\varphi\  ;\    \N (\phi)-W_h , v_h) } { \| v_ h   \| _{H ^1 (\Omega,T_h) } }=\frac { A_h(u\  ;\  u-W_h , v_h) } { \| v_ h   \| _{H ^1 (\Omega,T_h) } }\lesssim \frac{\|u-W_h\|_{H ^1 (\Omega,T_h)}\|v_h\|_{H ^1 (\Omega,T_h)}}{\|v_h\|_{H ^1 (\Omega,T_h)}},
\end{align*}
and so by taking the supremum for all $v_h\in V^q_h(\Omega)$ and applying Lemma \ref{lemma:infsup} we infer that
\begin{align*}
    \|\N (\phi)- W_h\|_{H ^1 (\Omega,T_h)}\lesssim \| u - W_h\|_{H ^1 (\Omega,T_h)}\leq C(u) h^r,
\end{align*}
where in the last inequality we used the properties of the approximation operator $I_h$. Hence $\N (\phi)\in\JJ$. It remains to show that $\N$ is a contraction. Indeed, for $\phi,\,\psi\in\JJ$ we have that
\begin{align*}
    A_h(\phi ;  \N (\phi)   - W_h , v_h)  &=  A_h(u ;  u - W_h , v_h) \\
    A_h(\psi ;  \N (\psi)   - W_h , v_h)  &=  A_h(u ;  u - W_h , v_h) ,
\end{align*}
for all $v_h\in V_h^q(\Omega)$. By subtracting the second from the first line we have that
\begin{align*}
    A_h(\phi ;  \N (\phi)   - \N(\psi) , v_h)=A_h(\psi ;  \N (\psi)   - W_h , v_h)-A_h(\phi ;  \N (\psi)   - W_h , v_h),
\end{align*}
hence we infer that
\begin{align*}
    \frac { A_h(\varphi\  ;\    \N (\phi)-\N (\psi) , v_h) } { \| v_ h   \| _{H ^1 (\Omega,T_h) } }&=\frac { A_h(\psi ;  \N (\psi)   - W_h , v_h)-A_h(\phi ;  \N (\psi)   - W_h , v_h) } { \| v_ h   \| _{H ^1 (\Omega,T_h) } }\\
    &\lesssim \|\nabla\phi-\nabla\psi\|_\infty \| \N (\psi) - W_h\|_{H ^1 (\Omega,T_h)}.
\end{align*}
Taking again the supremum over $v_h\in V^q_h(\Omega)$ and applying Lemma \ref{lemma:infsup}, we deduce that
\begin{align*}
    \| \N (\psi) - \N(\psi)\|_{H ^1 (\Omega,T_h)}&\lesssim \|\nabla\phi-\nabla\psi\|_\infty \| \N (\psi) - W_h\|_{H ^1 (\Omega,T_h)}\\
    &\lesssim C_\ast(u)h^r\|\nabla\phi-\nabla\psi\|_{L^2(\Omega)}\\
    &\leq c(h)\|\phi-\psi \|_{H ^1 (\Omega,T_h)},
\end{align*}
where  we used the fact that $\N(\psi)\in\JJ$ and a proper inverse inequality. This proves our claim and hence $\N$ has a unique fixed point $\phi_*= \N (\phi_*) $ which solves the discrete problem. Furthermore,  for $u_h:=\phi_\ast =  \N (\phi_*)  $ belongs to $\JJ,$ and satisfies the error estimate
\begin{equation*}\label{final_err_est}
\begin{aligned}
\| u - u_h \| _{H ^1 (\Omega,T_h) }  \leq \| u_ h - W_h \| _{H ^1 (\Omega,T_h) }  + \| u  - W_h \| _{H ^1 (\Omega,T_h) }  \leq \tilde C (u ) h^r  \, .
 \end{aligned}
\end{equation*}


\section{Computational Results}\label{computations}

In this section, employing the proposed penalty terms of (\ref{eq:penalty}) and (\ref{eq:penalty2}), plane deformations for elastic materials 
described by convex strain energy functions are computed in order to examine the numerical validity of the proposed schemes. 
Let $y :\Omega \subset \R^2 \rightarrow \Omega^*$ denote a 
planar deformation, where $\Omega$, $\Omega^*$ are the current and deformed states respectively. Specifying boundary conditions the 
deformation of the body is computed by minimising the total potential energy
\begin{align}
\min_{\substack{ y \in W^{1,p}(\Omega) \\ y = y_0 \text{ on } \partial \Omega} } \mathcal{E}( y) =  \min_{\substack{ y \in W^{1,p}(\Omega) \\ y = y_0 \text{ on } \partial \Omega} }
\int_\Omega W(\nabla y) dx,
\label{sim:energMin}
\end{align}
where $W$ is the strain energy function. Note that if $y_0$ corresponds to a homogeneous deformation, i.e. $y_0(x) = F_0 x$, $F_0 \in R^{2 \times 2}$ is constant, then for 
convex strain energy functions eq.~(\ref{sim:energMin}) is minimised  for $y = 
y_0$ and every other possible minimiser has greater or equal total potential energy. Here we 
choose $y_0(x) = F_0 x$ and $F_0 = diag\{1, 1 + \beta\} $, 
 $\beta >0$ ($\beta<0$) which implies uniaxial tension (compression).
 Then, given a strain energy $W$, we  minimise the discrete energy of eq.~(\ref{eq:energy_discrete}) for $y_h \in V^1_h(\Omega)$, i.e. piecewise polynomial functions of first degree. The discretization 
 is provided by the FEniCS project \cite{alnaes2015fenics} and details for the minimisation algorithm are given in \cite{grekas2022approximations}.
 We choose the convex  strain energy functions $W(F) = |F|^4$ and $W(F) = |F|^6$, $F = \nabla y$ is the deformation gradient. For the former energy $y_0$ corresponds to uniaxial tension, $\beta =0.1$, and for the latter boundary conditions  of uniaxial compression, $\beta = -0.1$. 
 For each case we compute the $|\cdot|_{L^1(\Omega)}$ and $|\cdot|_{W^{1,1}(\Omega)}$ errors between the exact solutions $y_0$ and the computed minimisers $y_h$ varying the penatly parameter $\alpha$ while increasing mesh resolution, 
 Fig~\ref{fig_errors}. Here it is evident that the penalty term of eq~\eqref{eq:penalty} captures 
 the homogeneous minimiser for both tension and compression even for small values of the 
 penalty parameter $\alpha$. Instead, for the penalty term of eq~\eqref{eq:penalty2} 
 the error is apparently larger when  $\alpha \le 160$, 
noting that these deformations do not correspond to any minimiser of 
the  continuous problem, because the total  potential energy is smaller than the homogeneous 
 deformation and this cannot hold for any  deformation in $W^{1,p}(\Omega)$. 
To quantify further these errors the ratio between reference and deformed triangle areas  are illustrated 
from the quantity $\det \nabla y$ in the reference configuration (undeformed mesh). 
 Note that in case of uniaxial tension 
 $\det \nabla y_0 = (1 +\beta)=1.1$ is constant. 
  For the penalty \eqref{eq:penalty2}, the 
 penalty parameter should be $\alpha \ge 160$, otherwise $\det \nabla y_h <1$
 for every triangle which is inconsistent with uniaxial tension, see Figs.~\ref{fig:convex_a1_td},
 \ref{fig:convex_a4_td}. Intuitively, one
 can expect these solutions from the fact that $|\nabla y_h|^4$ is minimised when
 $\nabla y_h = \bf{0}$ while for small penalty parameter values discontinuities are allowed implying 
 that 
 the dominant term is $|\nabla y_h|^4$. On the contrary, employing the penalty 
 \eqref{eq:penalty}, even for small values of $\alpha$, e.g. $\alpha=20$, 
 $\nabla y_h$ values are almost constant approaching $\nabla y_0$, see 
 Figs.~\ref{fig:quasiconvex_a1_td} and \ref{fig:exact_td}.
 Under uniaxial compression it is noteworthy that  failure of convergence occurred
 for the penalty \eqref{eq:penalty2} for $\alpha=20$ and $40$. Instead 
 penalty \eqref{eq:penalty} shows a more robust convergence behaviour.

 The above observations are in agreement with our theoretical results, where we have proved that both penalties 
 \eqref{eq:penalty}, \eqref{eq:penalty2} are suitable  for convex energies when the penalty parameter
 $\alpha$ is large enough to ensure coercivity of the discrete energy functional. \\
 
 Furthermore, to confirm the error estimates derived in Section 4
 we consider two problems with smooth minimisers. 
 The energy functional is of the form
 \begin{align}
 \min_{\substack{ y \in W^{1,p}(\Omega) \\ y = y_0 \text{ on } \partial \Omega} } \mathcal{E}( y) =  \min_{\substack{ y \in W^{1,p}(\Omega) \\ y = y_0 \text{ on } \partial \Omega} }
\int_\Omega |\nabla y(x)|^p + y(x)\cdot f(x) dx,
\label{eq:error_estimates}
 \end{align}
 where $p=2$ or $p=4$. The vector valued function $f$ is chosen in a way such that the unique minimizer of eq.~(\ref{eq:error_estimates}) is $y_0$ where
 \begin{align}
 y_0(x) = x + (0.1 x_1, 0.1 \sin\left( (x_1 +x_2)\pi \right).
 \label{eq:error_exact}
 \end{align}
We consider piecewise linear finite element spaces. In Table 1 the error $||y_h -y_0||_{W^{1,2}, T_h}$ in the case $p=2$ is displayed, where 
$y_h$ is the computed minimizer for various values of the penalty parameter $\alpha$, while for 
comparison we performed simulations with continuous elements (CG). 
We conclude our computational experiments by minimizing  (\ref{eq:error_estimates}) for 
$p=4$ again with a function $f$ such that the unique minimizer is given from  
(\ref{eq:error_exact}).
In both cases the experimental order is 
one, as expected, see Tables 1, 2 and Figure 3. 

 \clearpage
 
\begin{figure}
\centering
 \subfloat[Tension: $W = |F|^4$]{\includegraphics[width=0.5\linewidth]{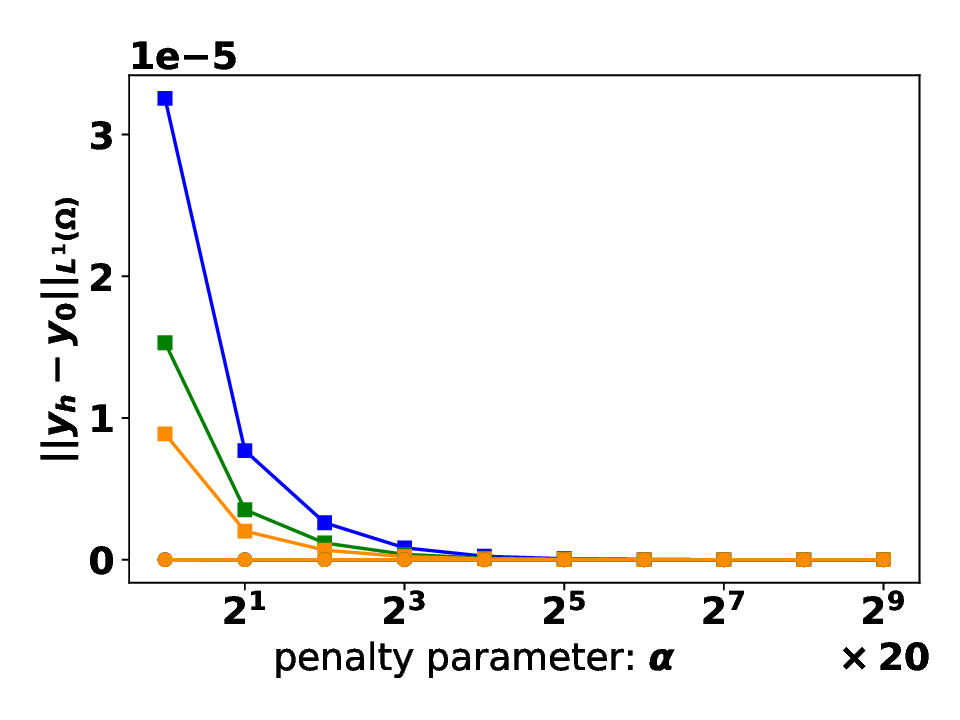}
 \label{fig:fn1a} 
 }
\subfloat[Tension: $W = |F|^4$]{ \includegraphics[width=0.5\linewidth]{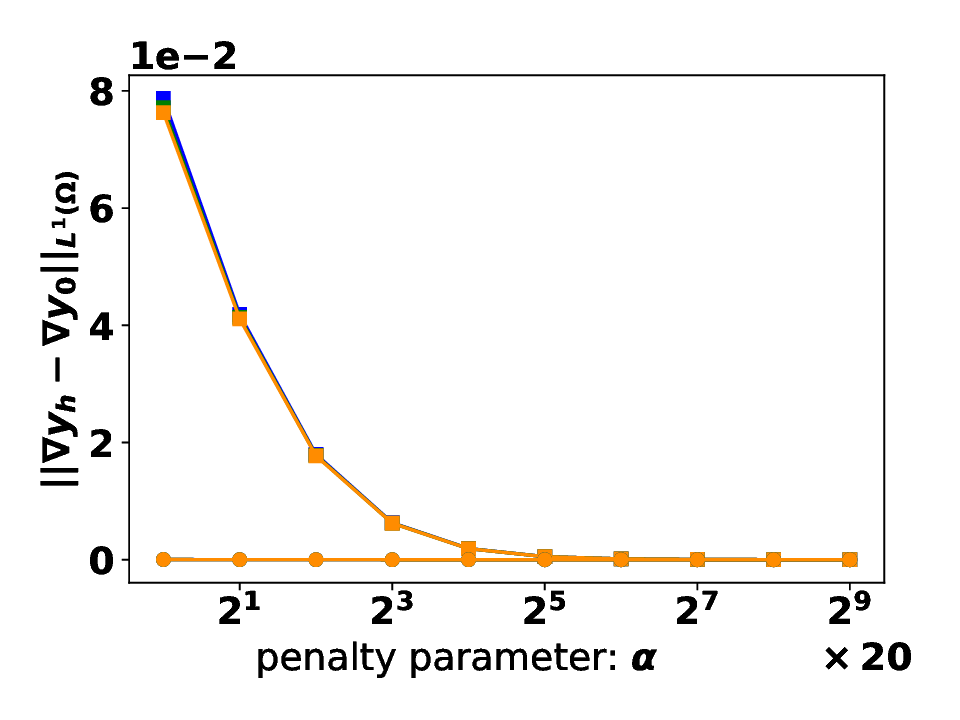}
 \label{fig:fn1b} }
  \\
  \subfloat[Compression:$W = |F|^6$]{ \includegraphics[width=0.5\linewidth]{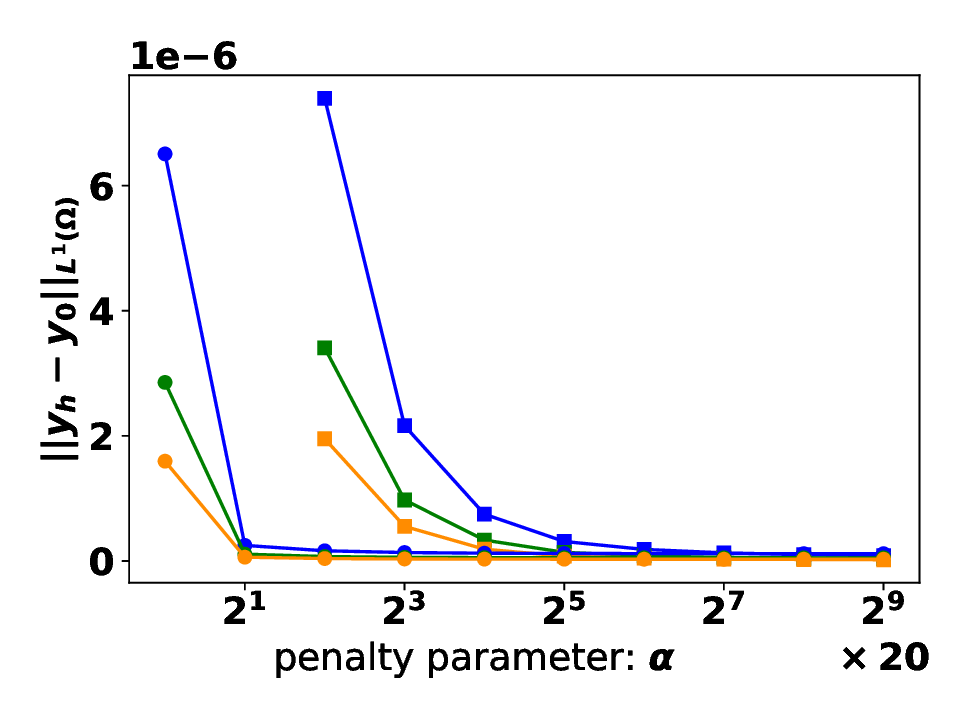}
  \label{fig:fn1c} }
 \subfloat[Compression:$W = |F|^6$]{ \includegraphics[width=0.5\linewidth]{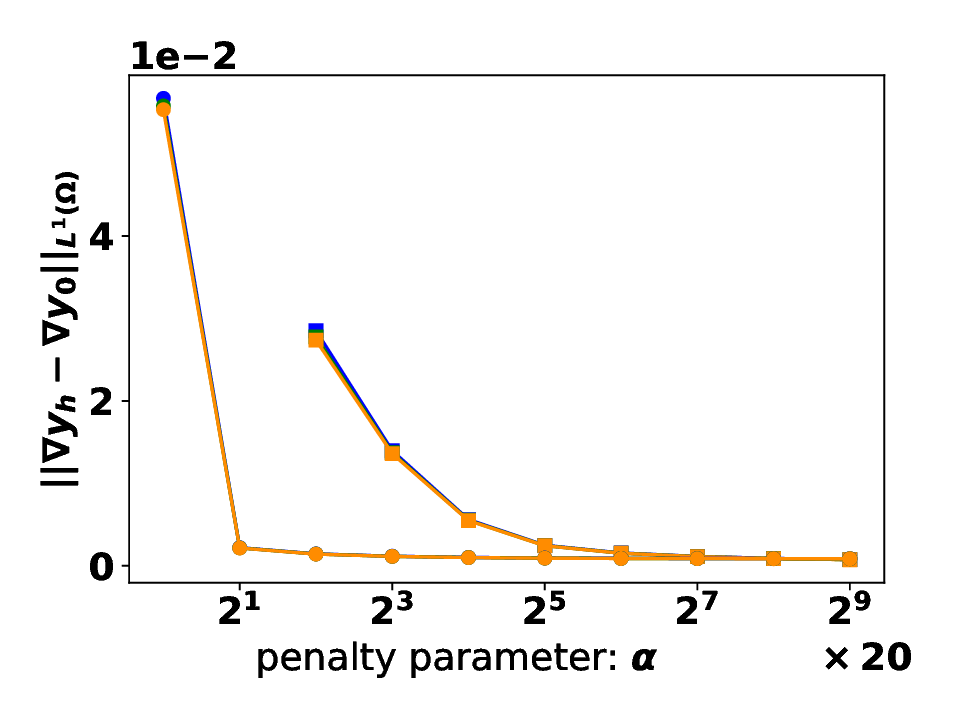}
 \label{fig:fn1d} 
  }
  \caption{
  Comparing the accuracy  of the two proposed stabilisation penalty terms \eqref{eq:penalty} (circles) and 
  \eqref{eq:penalty2} (squares) for the convex strain energies 
  $W(F) = |F|^4$ and $W(F) = |F|^6$.
  For the former energy the material is subjected to 10\% uniaxial tension, 
  while for the latter to 10\%   uniaxial compression. 
  Horizontal axis: values of the penalty parameter $\alpha$. 
  Vertical axes: $|u_h - u_e|_{L^1(\Omega)}$ and  $|u_h - u_e|_{W^{1,1}(\Omega)}$ errors 
  for various mesh resolutions, specifically for 1024 (blue), 2034 (green) and 4096 (orange) triangles,
  where   $u_h$ denotes the numerical solutions and $u_e$ the exact homogeneous minimizers. 
  In \protect\subref{fig:fn1a} and 
  \protect\subref{fig:fn1c} circular error belong in the range of 
  $10^{-8}$ and $10^{-9}$, while
  in \protect\subref{fig:fn1b} and  \protect\subref{fig:fn1d} circular errors are of the order 
  $10^{-6}$, which is determined by the numerical minimization criterion, i.e. stop iterations when 
  $\sup_{K \in T_h}\norm{\nabla u_h}_{L^{\infty}(K)} < 10^{-5}.$ In \protect\subref{fig:fn1c} and 
  \protect\subref{fig:fn1d} the scheme employing the penalty \eqref{eq:penalty2} did not converge 
  for the values $\alpha =20$ and  $40$.
  } 
  \label{fig_errors}
\end{figure}

\begin{figure}
\centering
 \subfloat[]{\hspace{-0.75cm} \includegraphics[width=0.55\linewidth]{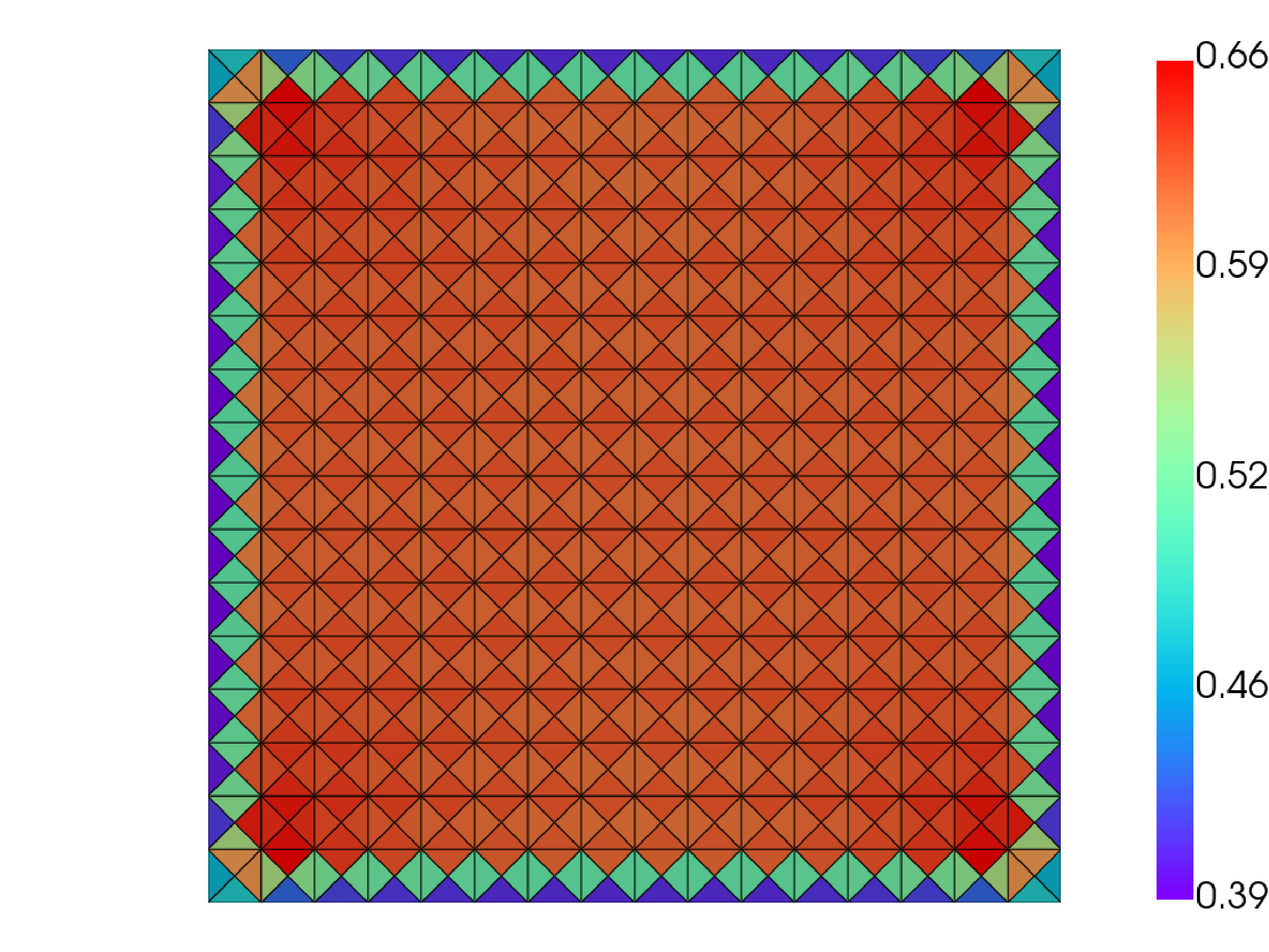} 
\label{fig:convex_a1_td} }
\subfloat[]{\hspace{-0.35cm} \includegraphics[width=0.55\linewidth]{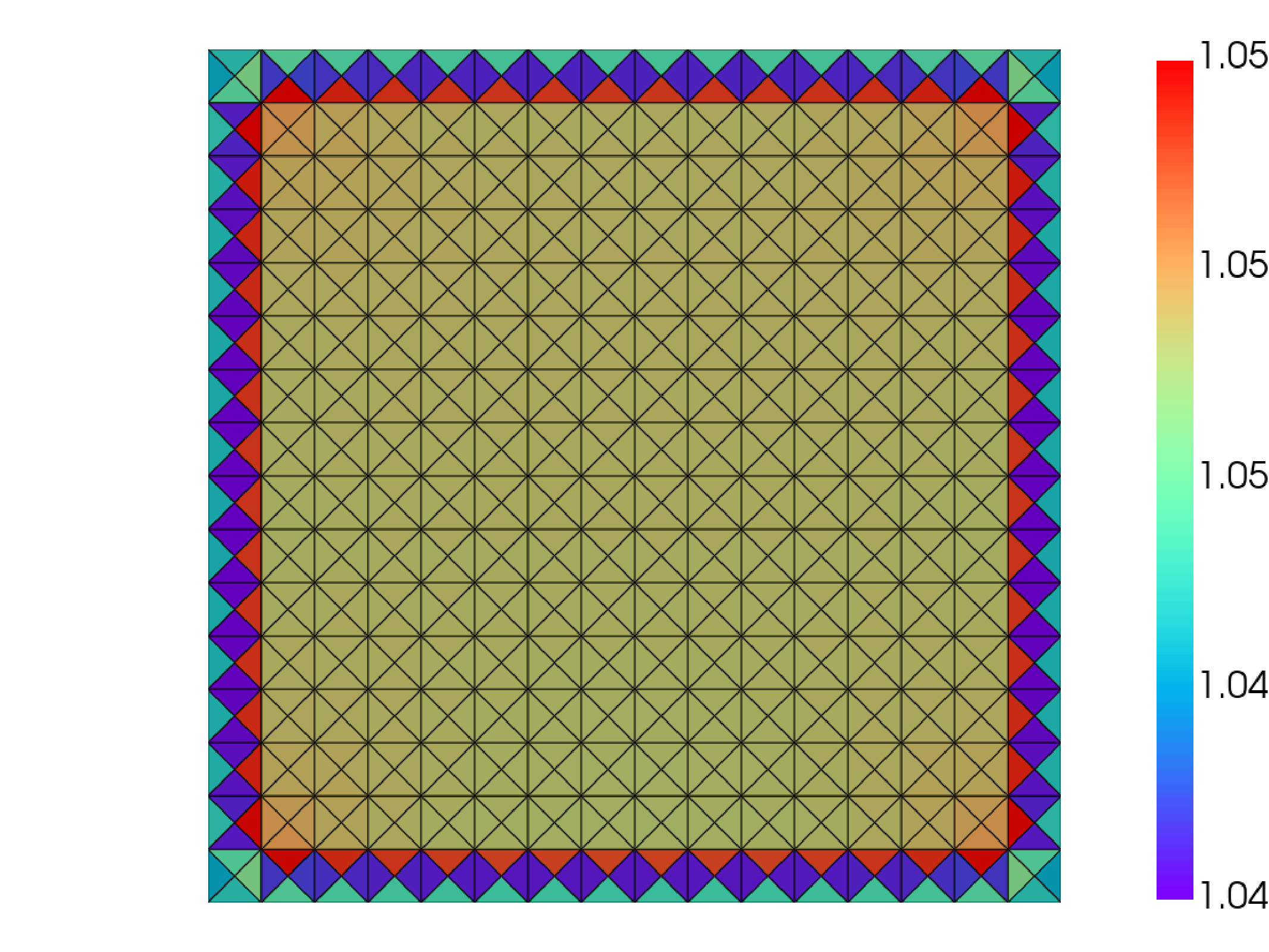}
\label{fig:convex_a4_td} }
 \\
\subfloat[]{\hspace{-0.7cm}\includegraphics[width=0.55\linewidth]{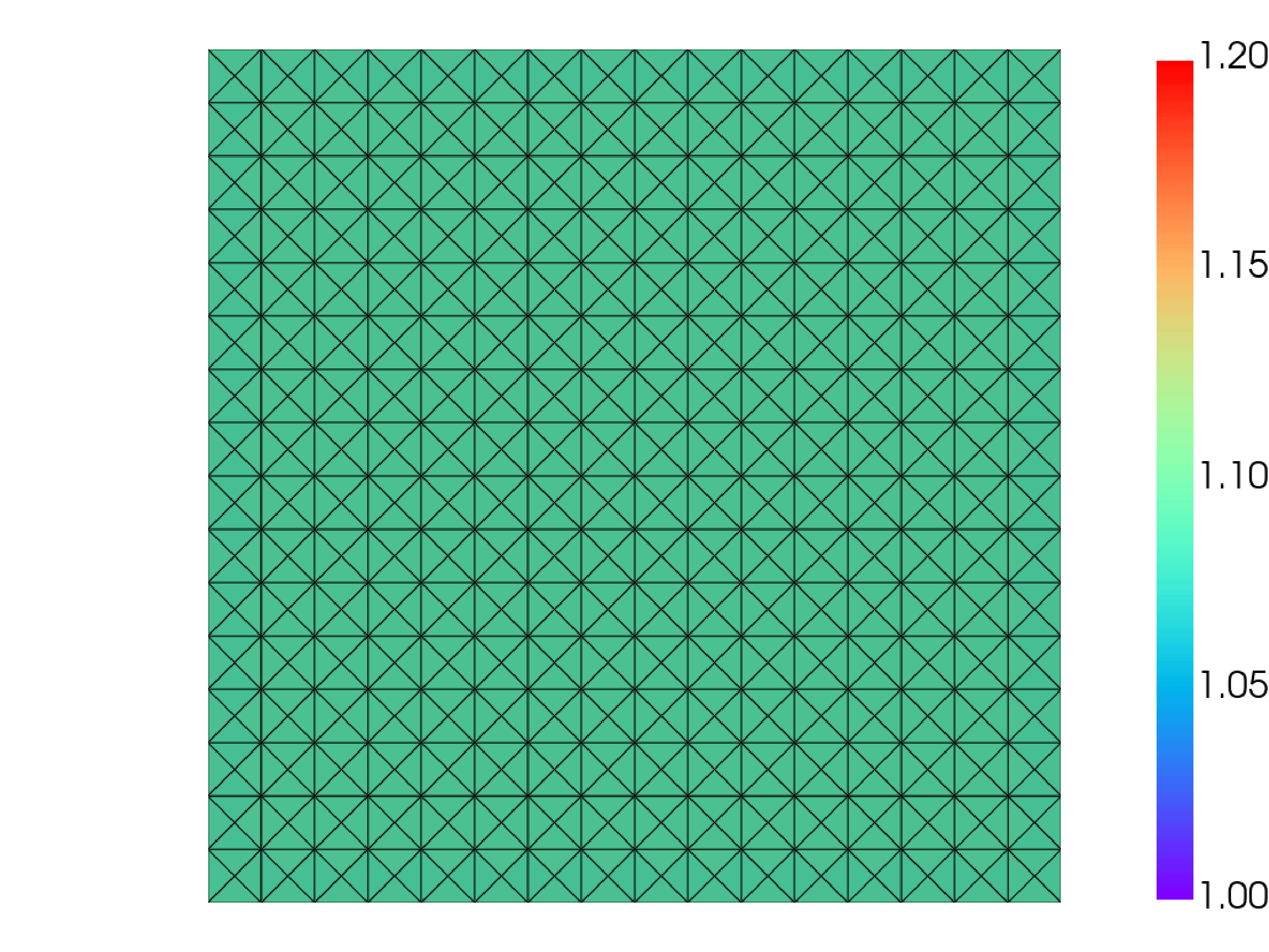} 
\label{fig:quasiconvex_a1_td} }
 \subfloat[]{\hspace{-0.2cm}\includegraphics[width=0.55\linewidth]{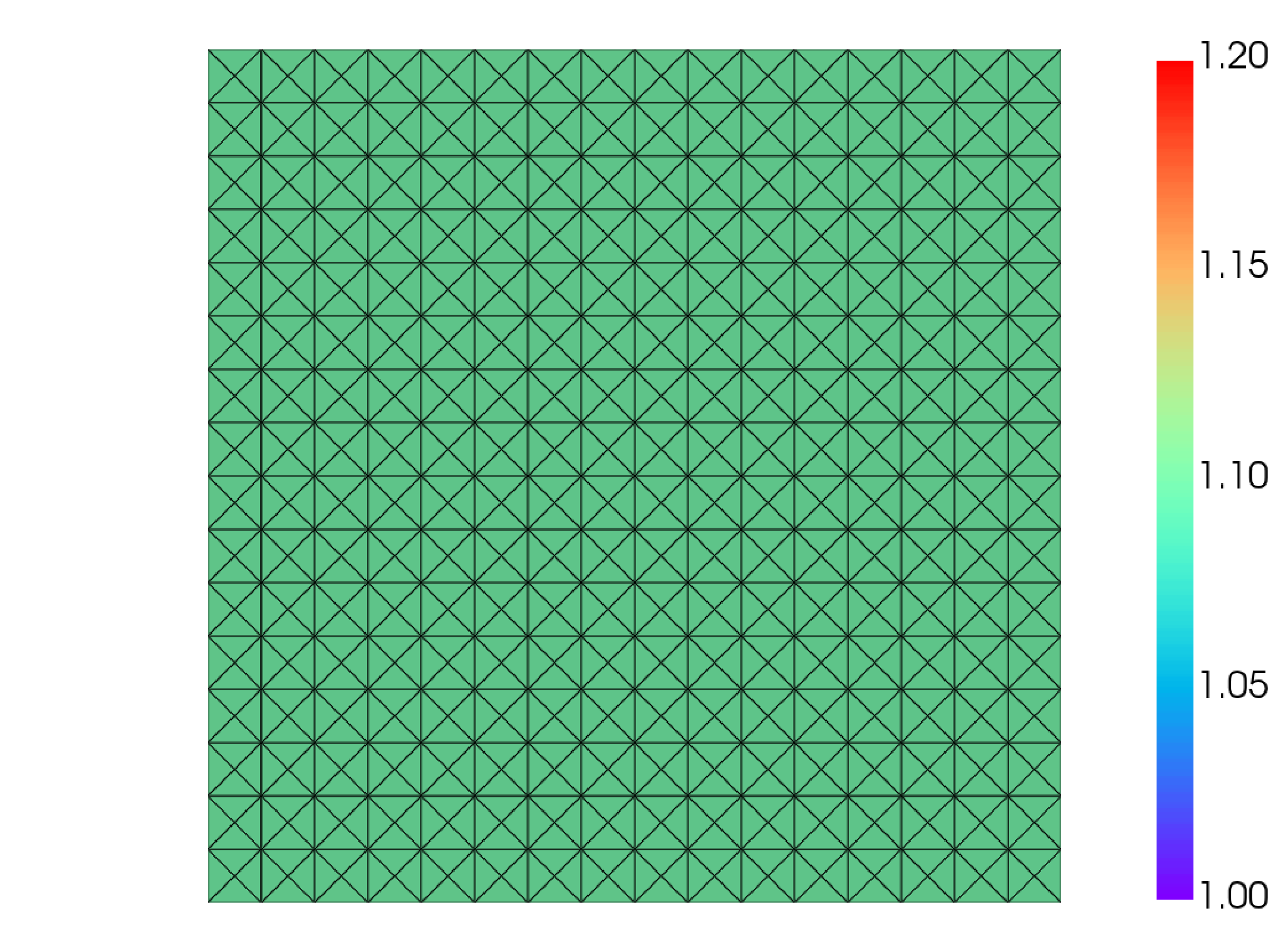} 
\label{fig:exact_td} }
 \caption{ Reference to deformed triangle area ($\det \nabla{y_h}$) in the reference configuration under uniaxial tension (10\% strain). 
 Employing the penalty of eq.~\eqref{eq:penalty2}, $\det \nabla{y_h}$ for the computed minimizers is illustrated: \protect\subref{fig:convex_a1_td} when $\alpha=20$ and 
\protect\subref{fig:convex_a4_td} when $\alpha =160$ (see blue squares of Figs.~\ref{fig:fn1a} and \ref{fig:fn1b} at $\alpha =20, 160$).
\protect\subref{fig:quasiconvex_a1_td}: Area ratio for the penalty of eq.\eqref{eq:penalty} with $\alpha=20$.
\protect\subref{fig:exact_td} $\det \nabla{I_h y_0}$, 
$I_h:W^{1,p}(\Omega) \rightarrow V_h$, the standard interpolation operator. 
}
 \centering
 \label{fig:tension_detF}
\end{figure}

\begin{figure}[h!]
\centering
 {\includegraphics[width=0.90\linewidth]{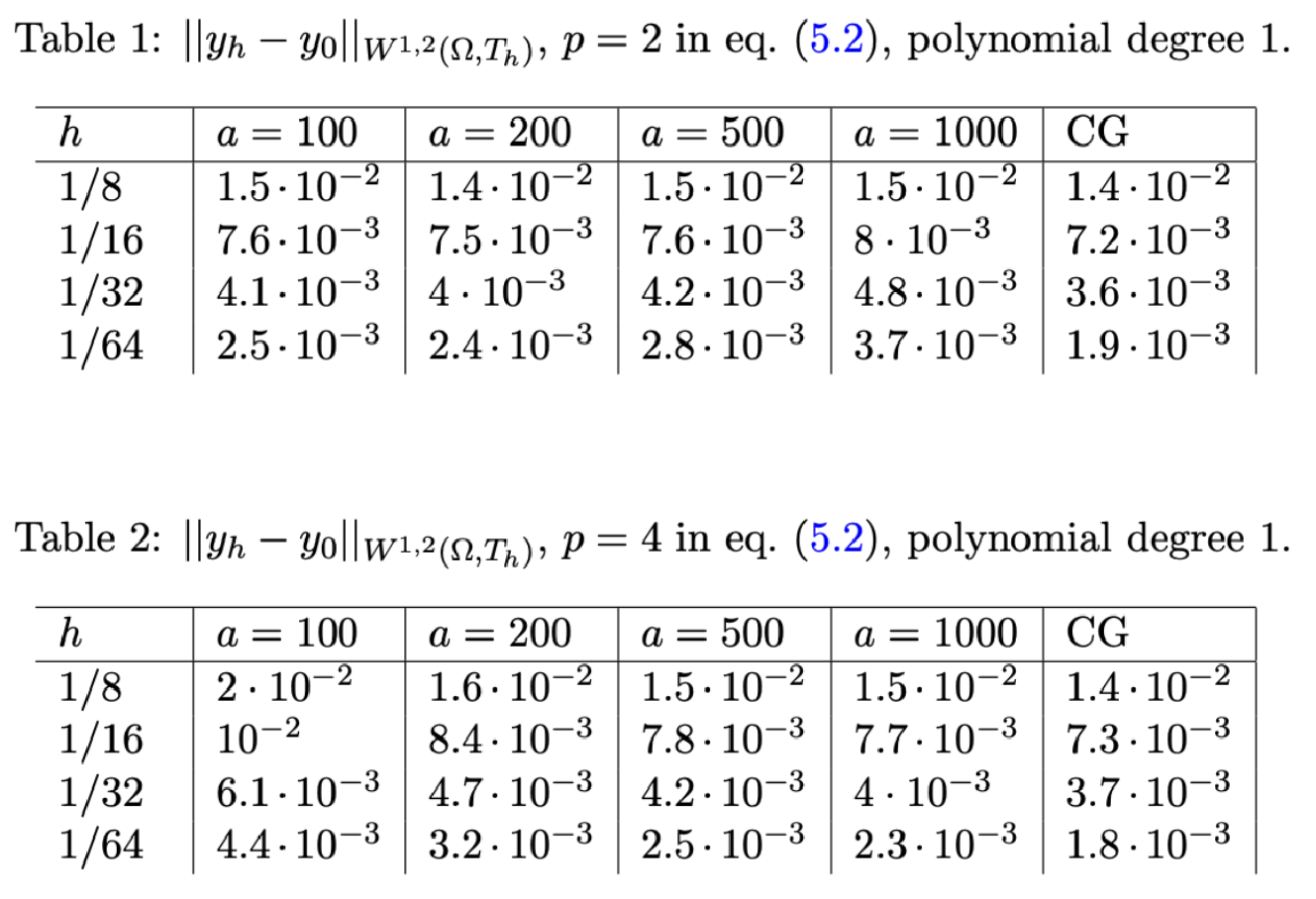}
 \label{fig:p2tablea} 
 }
  \label{fig_table_errorsa}
\end{figure}

\begin{figure}[h!]
\centering
 \subfloat[p=2]{\includegraphics[width=0.55\linewidth]{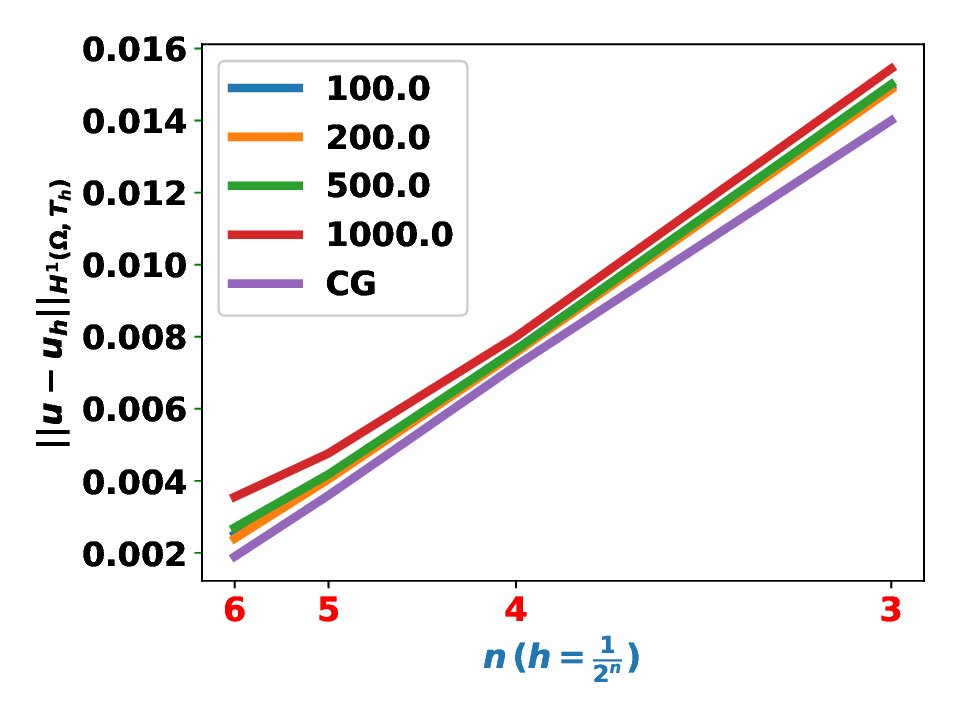}
 \label{fig:p2tableb} 
 }
\subfloat[p=4]{ \includegraphics[width=0.55\linewidth]{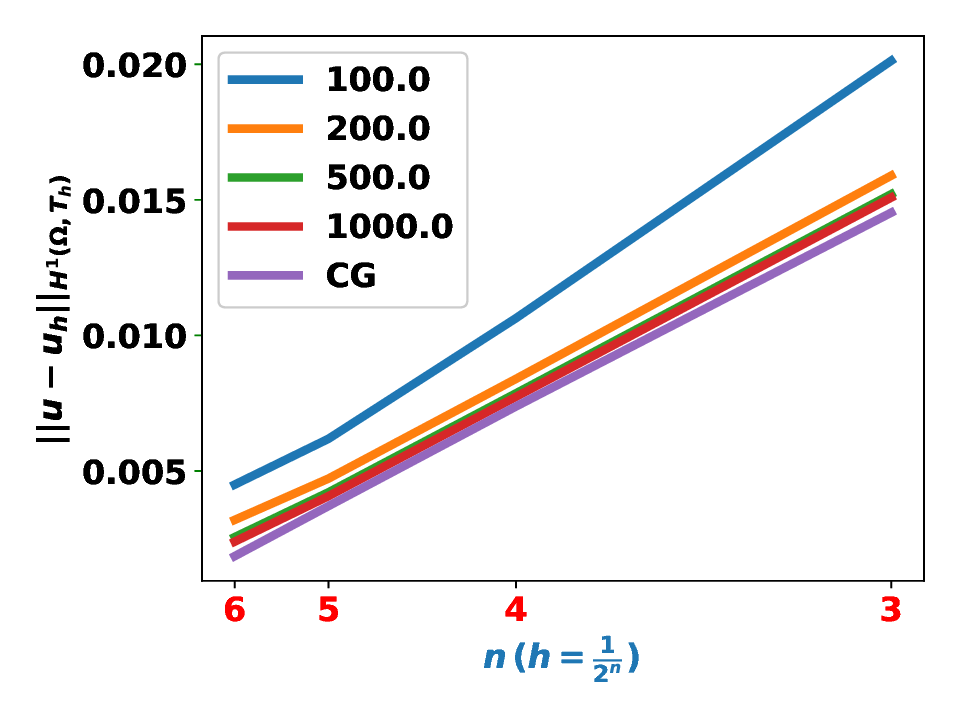}
 \label{fig:p2tablec} }
  \caption{Convergence behaviour in $W^{1,2}(\Omega,T_h)$ of the error over discretization parameter $h.$  Table 1 and plot (a) correspond to $p=2$ and 
  Table 2 and plot (b) correspond to $p=4$
  in (\ref{eq:error_estimates}) for various values of the penalty parameter $a.$
  } 
  \label{fig_table_errorsb}
\end{figure}

\section*{Acknowledgements} The research of A.V. was funded in whole by the Austrian Science Fund (FWF) projects I4354 and I5149. The research of K.K. is partially funded by EPSRC, United Kingdom EP/X038998/1 grant and receives support from the Dr Perry James (Jim) Browne Research Centre, University of Sussex.

\clearpage
\bibliographystyle{alpha}
\bibliography{references.bib}

\newcommand{\etalchar}[1]{$^{#1}$}
\begin{thebibliography}{GKMR22}

\bibitem[ABH{\etalchar{+}}15]{alnaes2015fenics}
Martin Aln{\ae}s, Jan Blechta, Johan Hake, August Johansson, Benjamin Kehlet,
  Anders Logg, Chris Richardson, Johannes Ring, Marie~E Rognes, and Garth~N
  Wells.
\newblock The fenics project version 1.5.
\newblock {\em Archive of numerical software}, 3(100), 2015.

\bibitem[AFP00]{ambrosio2000functions}
Luigi Ambrosio, Nicola Fusco, and Diego Pallara.
\newblock {\em Functions of bounded variation and free discontinuity problems}.
\newblock Oxford university press, 2000.

\bibitem[Arn82]{Arnold_DG}
Douglas~N. Arnold.
\newblock An interior penalty finite element method with discontinuous
  elements.
\newblock {\em SIAM Journal on Numerical Analysis}, 19(4):742--760, 1982.

\bibitem[BBN17]{bartels2017bilayer}
S{\"o}ren Bartels, Andrea Bonito, and Ricardo~H Nochetto.
\newblock Bilayer plates: Model reduction, {$\Gamma$}-convergent finite element
  approximation, and discrete gradient flow.
\newblock {\em Communications on Pure and Applied Mathematics}, 70(3):547--589,
  2017.

\bibitem[BNN21]{BNNtogkas:2021}
Andrea Bonito, Ricardo~H. Nochetto, and Dimitrios Ntogkas.
\newblock D{G} approach to large bending plate deformations with isometry
  constraint.
\newblock {\em Math. Models Methods Appl. Sci.}, 31(1):133--175, 2021.

\bibitem[BO09]{buffa2009compact}
Annalisa Buffa and Christoph Ortner.
\newblock Compact embeddings of broken sobolev spaces and applications.
\newblock {\em IMA journal of numerical analysis}, 29(4):827--855, 2009.

\bibitem[BS07]{brenner2007mathematical}
Susanne Brenner and Ridgway Scott.
\newblock {\em The mathematical theory of finite element methods}, volume~15.
\newblock Springer Science \& Business Media, 2007.

\bibitem[Dac07]{dacorogna2007direct}
Bernard Dacorogna.
\newblock {\em Direct methods in the calculus of variations}, volume~78.
\newblock Springer Science \& Business Media, 2007.

\bibitem[DH85]{DafHrusa_1985}
Constantine~M. Dafermos and William~J. Hrusa.
\newblock Energy methods for quasilinear hyperbolic initial-boundary value
  problems. {A}pplications to elastodynamics.
\newblock {\em Arch. Rational Mech. Anal.}, 87(3):267--292, 1985.

\bibitem[DPE10]{di2010discrete}
Daniele Di~Pietro and Alexandre Ern.
\newblock Discrete functional analysis tools for discontinuous galerkin methods
  with application to the incompressible navier-stokes equations.
\newblock {\em Mathematics of Computation}, 79(271):1303--1330, 2010.

\bibitem[GKMR22]{grekas2022approximations}
Georgios Grekas, Konstantinos Koumatos, Charalambos Makridakis, and Phoebus
  Rosakis.
\newblock Approximations of energy minimization in cell-induced phase
  transitions of fibrous biomaterials: {$\Gamma$}-convergence analysis.
\newblock {\em SIAM Journal on Numerical Analysis}, 60(2):715--750, 2022.

\bibitem[GP99]{gobb_prohl_1999}
Matthias~K Gobbert and Andreas Prohl.
\newblock A discontinuous finite element method for solving a multiwell
  problem.
\newblock {\em SIAM journal on numerical analysis}, 37(1):246--268, 1999.

\bibitem[KLST20]{koumatos2020existence}
Konstantinos Koumatos, Corrado Lattanzio, Stefano Spirito, and Athanasios~E
  Tzavaras.
\newblock Existence and uniqueness for a viscoelastic {K}elvin-{V}oigt model
  with nonconvex stored energy.
\newblock {\em arXiv preprint arXiv:2012.10344}, 2020.

\bibitem[KP03]{karakashian2003posteriori}
Ohannes~A Karakashian and Frederic Pascal.
\newblock A posteriori error estimates for a discontinuous {G}alerkin
  approximation of second-order elliptic problems.
\newblock {\em SIAM Journal on Numerical Analysis}, 41(6):2374--2399, 2003.

\bibitem[KV21]{Koumatos_Vikelis_2021}
Konstantinos Koumatos and Andreas~P. Vikelis.
\newblock A-quasiconvexity, {G}årding inequalities, and applications in {PDE}
  constrained problems in dynamics and statics.
\newblock {\em SIAM Journal on Mathematical Analysis}, 53(4):4178--4211, 2021.

\bibitem[Mak93]{makr_1993}
Charalambos~G. Makridakis.
\newblock Finite element approximations of nonlinear elastic waves.
\newblock {\em Math. Comp.}, 61(204):569--594, 1993.

\bibitem[MIB96]{MakrIB_1996}
Ch. Makridakis, F.~Ihlenburg, and I.~Babu\v{s}ka.
\newblock Analysis and finite element methods for a fluid-solid interaction
  problem in one dimension.
\newblock {\em Math. Models Methods Appl. Sci.}, 6(8):1119--1141, 1996.

\bibitem[MMR14]{makridakis2014atomistic}
Charalambos Makridakis, Dimitrios Mitsoudis, and Phoebus Rosakis.
\newblock On atomistic-to-continuum couplings without ghost forces in three
  dimensions.
\newblock {\em Applied Mathematics Research eXpress}, 2014(1):87--113, 2014.

\bibitem[OS07]{ortner2007discontinuous}
Christoph Ortner and Endre S{\"u}li.
\newblock Discontinuous {G}alerkin finite element approximation of nonlinear
  second-order elliptic and hyperbolic systems.
\newblock {\em SIAM Journal on Numerical Analysis}, 45(4):1370--1397, 2007.

\bibitem[Sch74]{Schatz_1974}
Alfred~H. Schatz.
\newblock An observation concerning {R}itz-{G}alerkin methods with indefinite
  bilinear forms.
\newblock {\em Math. Comp.}, 28:959--962, 1974.

\bibitem[TEL06]{ten2006discontinuous}
Alex Ten~Eyck and Adrian Lew.
\newblock Discontinuous {G}alerkin methods for non-linear elasticity.
\newblock {\em International Journal for Numerical Methods in Engineering},
  67(9):1204--1243, 2006.

\bibitem[Zie89]{ziemer1989weakly}
William~P Ziemer.
\newblock Weakly differentiable functions, sobolev spaces and functions of
  bounded variation.
\newblock {\em Grad. Texts in Math.}, 120, 1989.

\end{thebibliography}

\end{document}